\documentclass[
a4paper,
11pt, 
DIV=14,
toc=bibliography, 
headsepline, 
parskip=half-,
abstract=true,
]{scrartcl}

\usepackage[english]{babel}
\usepackage[utf8]{inputenc}
\usepackage[T1]{fontenc}
\usepackage[english,cleanlook]{isodate}
\usepackage{lmodern}
\usepackage[babel=true]{microtype}
\usepackage{amsmath}
\usepackage{mathtools}
\usepackage{amssymb}
\usepackage{amsthm}
\usepackage{mathrsfs}
\usepackage{scrlayer-scrpage} 
\usepackage{enumitem}
\usepackage{tabularx}
\usepackage{tikz} 
\usepackage{tikz-3dplot}  

\ihead{M.\;B.\;Schulz, D.\;Wiygul}
\ohead{\csname @title\endcsname}
\setkomafont{pageheadfoot}{\footnotesize}

\providecommand{\R}{\mathbb{R}}
\providecommand{\Z}{\mathbb{Z}}
\providecommand{\N}{\mathbb{N}}
\providecommand{\Sp}{\mathbb{S}}
\providecommand{\T}{\mathbb{T}}
\providecommand{\hsd}{\mathscr{H}}
\providecommand{\st}{\, :\ }
\providecommand{\rs}{Z}
\providecommand{\refl}[1]{\underline{#1}}
\providecommand{\grp}[1]{\mathcal{#1}} 
\providecommand{\symdiff}{\mathbin{\triangle}}
\providecommand{\lune}{V}
\providecommand{\symgrp}[1]{G_{\mathrm{sym}}^{#1}}

\DeclarePairedDelimiter\sk{\langle}{\rangle}
\DeclarePairedDelimiter\interval{]}{[}
\DeclarePairedDelimiter\Interval{[}{[}

\DeclareMathOperator{\genus}{genus}
\DeclareMathOperator{\Ogroup}{O}
\DeclareMathOperator{\Span}{Span}
\DeclareMathOperator{\ind}{ind}

\usetikzlibrary{plotmarks}
\makeatletter
\tikzset{
    scale plot marks/.is choice,
    scale plot marks/true/.style={},	
    scale plot marks/false/.code={
        \def\pgfuseplotmark##1{\pgftransformresetnontranslations\csname pgf@plot@mark@##1\endcsname}
    },
every mark/.append style={scale plot marks=false},
plus/.style={mark=+,mark size=2.25pt},
vdash/.style={mark=|,mark size=2.25pt},
hdash/.style={mark=-,mark size=2.25pt},
bullet/.style={mark=*,mark size=1.125pt},
}
\makeatother
\definecolor{FarbeA}{cmyk}{0,0.5,1,0}
\definecolor{FarbeB}{cmyk}{1,0.5,0,0}	

\tikzset{ 
equator/.style={very thick,FarbeA},
axis/.style={red,very thick},
}

\usepackage[initials]{amsrefs}  
\makeatletter
\renewcommand{\amsrefs@warning}[1]{}  
\makeatother

\usepackage[
pdftitle={A new family of minimal surfaces of even genus in the three-dimensional sphere},
pdfauthor={Mario B. Schulz, David Wiygul}, 
pdfborder={0 0 0},
]{hyperref}

\theoremstyle{plain}
\newtheorem{theorem}{Theorem}[section]
\newtheorem{lemma}[theorem]{Lemma}
\newtheorem{corollary}[theorem]{Corollary} 
\newtheorem{proposition}[theorem]{Proposition}
\newtheorem{conjecture}[theorem]{Conjecture}

\theoremstyle{definition}
\newtheorem{definition}[theorem]{Definition}
 
\theoremstyle{remark}
\newtheorem{remark}[theorem]{Remark}

\title{A new family of minimal surfaces of even genus in~the three-dimensional sphere}
\author{Mario B. Schulz and David Wiygul}
\date{\vspace*{-4ex}}
 
\newcommand\printaddress{{
\setlength{\parindent}{17pt}
\medskip
\hfill\printdate{30.07.2025}
\par
{\scshape Mario B. Schulz}
\newline 
Università di Trento, 
Dipartimento di Matematica, 
via Sommarive 14, 
38123 Povo, 
Italy
\newline
\textit{E-mail address:} 
\texttt{mario.schulz@unitn.it}
\par\medskip
{\scshape David Wiygul}
\newline 
Università di Trento, 
Dipartimento di Matematica, 
via Sommarive 14, 
38123 Povo, 
Italy
\newline
\textit{E-mail address:} 
\texttt{mario.schulz@unitn.it}
}}

\begin{document}

\maketitle
 
\begin{abstract}
We discover a family of closed, embedded minimal surfaces in the three-dimensional round sphere which includes new examples with low genus. 
The existence proof relies on an equivariant min-max procedure applied to a novel sweepout which is constructed by fusing the equatorial sphere with the Clifford torus. 
We determine the full symmetry groups of our surfaces, prove lower bounds on their Morse indices, and show that they are geometrically distinct from all previously known examples. 
\end{abstract}

\section{Introduction}

A minimal surface in a given ambient manifold is a critical point of the area functional, or equivalently, a $2$-dimensional submanifold with vanishing mean curvature. 
While there are no closed minimal surfaces in Euclidean space $\R^3$, the sphere $\Sp^3$ contains a rich variety of compact, embedded examples without boundary. 
The simplest are the totally geodesic equatorial sphere $\Sp^2$ and the flat Clifford torus $\T^2$. 
Almgren \cite{Almgren1966} and Calabi \cite{Calabi1967} proved independently that $\Sp^2$ is unique up to congruence in the class of closed, immersed minimal surfaces of genus~$0$ in $\Sp^3$.  
Brendle \cite{Brendle2013} (see also his survey \cite{Brendle2013survey}) proved Lawson's \cite{Lawson1970conj} conjecture stating that $\T^2$ is unique up to congruence in the class of closed, embedded minimal surfaces of genus~$1$ in $\Sp^3$.  
Beyond these two examples, Lawson~\cite{Lawson1970} constructed minimal surfaces $\xi_{m,k}\subset\Sp^3$ of genus $mk$ for arbitrary $m,k\in\N$.
When $k$ is large, $\xi_{m,k}$ can be interpreted as a desingularization 
of $m+1$ great spheres intersecting along a common equator at equal angles. 
In this sense, the Lawson surfaces are spherical analogues of the complete
Karcher--Scherk \cite{KarcherScherk} towers
in $\R^3$, which desingularize planes intersecting along a line
(although \cite{Lawson1970} preceded
Karcher's \cite{KarcherScherk} generalization
of Scherk's \cite{Scherk1835} classical singly periodic minimal surface).

Lawson's construction proceeds by finding a minimal disc solving a well-selected Plateau problem in $\Sp^3$ and extending it to a closed surface by suitable reflections. 
Variations on this technique have been successfully applied in 
\cite{KarcherPinkallSterling,ChoeSoret-tordesing,BaiWangWang,KW-tormore}.
Many other families of minimal surfaces in $\Sp^3$ have been constructed by gluing methods in 
\cite{KapouleasYang,
Kapouleas-sphdbl,
W-torstack,
KW-tordesing,
KapouleasMcGrath-sphdbl,
KapouleasMcGrath2020,
KapouleasMcGrath-tordbl}.
A variety of equivariant min-max constructions were proposed in \cite{PittsRubinstein1988}, but none were carried out until the work \cite{Ketover2016Equivariant,KetoverMarquesNeves2020}; the article 
\cite{KetoverFlipping} contains further min-max examples, including ones not discussed by \cite{PittsRubinstein1988}. 
The Dorfmeister--Pedit--Wu \cite{DPW} method, based on an analogue of the Weierstrass representation, has also been a fruitful tool, with recent applications to minimal surfaces in $\Sp^3$ appearing in~\cites{HellerHellerTraizet-area,BobenkoHellerSchmitt-CMC,BobenkoHellerSchmitt-refl}.
Finally, minimal surfaces in $\Sp^3$ have been obtained in \cite{KarpukhinKusnerMcGrathStern} by equivariant optimization of the first nontrivial Laplace eigenvalue.
 
In this article, we construct a new family of minimal surfaces in $\Sp^3$. 
Our surfaces are spherical analogues of the celebrated
Costa--Hoffman--Meeks \cites{Costa, HoffmanMeeks-chm} minimal surfaces in $\R^3$.
The latter include for each $2\leq n\in\N$ a complete, embedded minimal surface of genus $n-1$ with one planar end and a pair of catenoidal ends.  
In our construction, the equatorial sphere $\Sp^2$ plays the role of the plane and the Clifford torus $\T^2$ plays the role of the catenoid (see Figure~\ref{fig:n=3}), and we likewise obtain a minimal surface for each $2\leq n\in\N$. 
In particular, we prove the existence of new embedded minimal surfaces in $\Sp^3$ with low genus. 
To achieve this, we employ $\grp{G}_n$-equivariant min-max methods.
For a definition of the group $\grp{G}_n$ and its action on $\Sp^3$ we refer to equation \eqref{eqn:group} in Section~\ref{sec:equivariant}.

\begin{figure}\centering 
\includegraphics[width=0.9\textwidth]{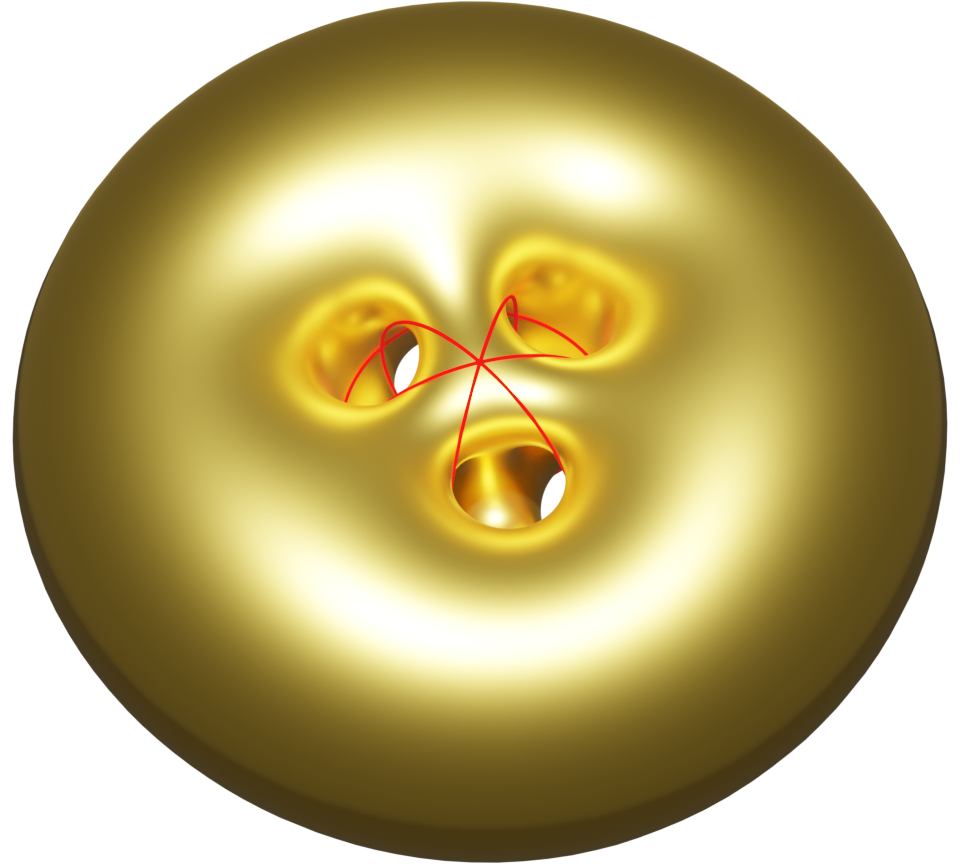}

\bigskip

\includegraphics[width=0.9\textwidth]{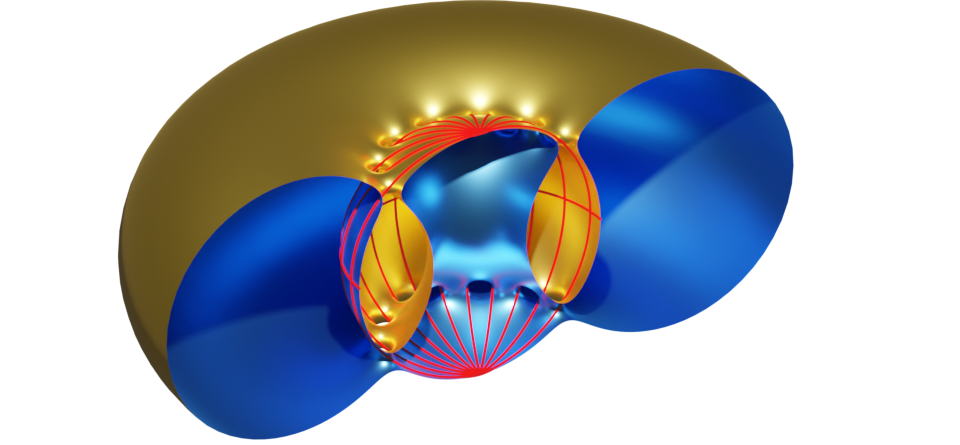}
\caption{Stereographic projection of $\Gamma_n$ for $n=3$ (top image) and $n=12$ (bottom image, cross-section), containing the equator $\Sp^1$ and the $n$ meridians $\xi_{k/n}$.}%
\label{fig:n=3}%
\end{figure}

\begin{theorem}\label{thm:main}
For each $2\leq n\in\N$ there exists a closed, embedded, $\grp{G}_n$-equivariant minimal surface $\Gamma_n\subset\Sp^3$ with the following properties. 
\begin{enumerate}[label={\normalfont(\roman*)},nosep]
\item\label{thm:main-genus} The genus of $\Gamma_n$ is equal to either $2n$ or $2n-2$. 
\item\label{thm:main-area} The area of $\Gamma_n$ is strictly between $2\pi^2$ and $2\pi^2+4\pi+\frac{1}{25}$.  
\item \label{thm:main-axes} $\Gamma_n$ contains the equator $\Sp^1$ as well as $n$ equally spaced meridians in the equatorial great sphere.  
\item\label{thm:main-symmetry} 
The group $\grp{G}_n$ of order $8n$ is the full symmetry group of $\Gamma_n$ for all $n\geq4$. 
\item\label{thm:main-index} The Morse index of $\Gamma_n$ is at least $4n-1$. 
\end{enumerate}
\end{theorem}

Theorem~\ref{thm:main} is proven in Section~\ref{sec:topology}, where we also show that $\Gamma_n$ is \emph{geometrically distinct} from any previously known embedded minimal surface in $\Sp^3$ for all $n\geq3$. 
The case of genus $2n-2$ in statement \ref{thm:main-genus} would be the result of a $\grp{G}_n$-equivariant topological degeneration during the min-max procedure, which we do not expect to happen;  
we conjecture $\Gamma_n$ to have genus $2n$ for all $2\leq n\in\N$. 
In the case $n=2$, we present evidence for a congruence with the Lawson surface $\xi_{2,2}$ of genus~$4$, 
whose symmetry group of order $144$ is much larger than the prescribed group $\grp{G}_2$ of order $16$. 
Equivariant min-max methods rarely give any control on the \emph{full} symmetry group of the resulting surfaces (cf.~\cite[Remark 1.2]{FranzKetoverSchulz}) and this fact makes statement \ref{thm:main-symmetry} a nontrivial claim. 
The Morse index growth rate stated in \ref{thm:main-index} is expected to be asymptotically sharp; 
indeed, in Conjecture~\ref{conj:index} we propose an exact formula for the index.

\paragraph{Acknowledgments.}
This project has received funding from the European Research Council (ERC) under the European Union's Horizon 2020 research and innovation programme (grant agreement No.~947923). 
We thank Alessandro Carlotto for many helpful comments and discussions.

\section{Structure of equivariant surfaces in the sphere}
\label{sec:equivariant}

In this section we introduce the group actions relevant for our construction. 
We then prove topological results which constrain the genus of any (not necessarily minimal) surface in $\Sp^3$ satisfying the symmetries we impose.  
We equip $\R^4$ with standard Cartesian coordinates $x_1,x_2,x_3,x_4$ and denote the $3$-dimensional unit sphere by 
\begin{align*}
\Sp^3&\vcentcolon=\{x\in\R^4\st x_1^2 + x_2^2 + x_3^2 + x_4^2 = 1\}.
\end{align*} 
On $\Sp^3$ we introduce spherical coordinates $\theta_1,\theta_2,\theta_3$, where $\theta_1\in\Interval{0,2\pi}$ is the azimuthal angle and $\theta_2,\theta_3\in\Interval{-\frac{\pi}{2},\frac{\pi}{2}}$ are latitudinal angles such that 
\begin{align*} 
x_1&=\cos(\theta_3)\cos(\theta_2)\cos(\theta_1),     \\
x_2&=\cos(\theta_3)\cos(\theta_2)\sin(\theta_1),     \\
x_3&=\cos(\theta_3)\sin(\theta_2),   \\
x_4&=\sin(\theta_3). 
\end{align*}
Using this convention, we understand the \emph{equatorial sphere} $\Sp^2$, its \emph{equator} $\Sp^1$, and the \emph{Clifford torus} $\T^2$ as follows (see Figure~\ref{fig:symmetry}): 
\begin{align*}
\begin{aligned}
\Sp^2&\vcentcolon=\{x\in\Sp^3\st x_4=0\}
&&=\{\theta_3=0\},\\
\Sp^1&\vcentcolon=\{x\in\Sp^3\st x_3=x_4=0\}
&&=\{\theta_2=\theta_3=0\},\\
\T^2&\vcentcolon=\{x\in\Sp^3\st x_1^2 + x_2^2 = x_3^2 + x_4^2\}
\!\!\!&&=\bigl\{\cos(\theta_2)\cos(\theta_3)=\tfrac{1}{\sqrt{2}}\bigr\}.
\end{aligned}
\end{align*}  
Given any $t\in\R$ we define
\begin{align*}
\Xi_t&\vcentcolon=
\bigl\{
x\in\Sp^3\st x_1\sin(t\pi)-x_2\cos(t\pi) = 0
\bigr\} 
=\bigl\{\theta_1=t\pi\bigr\},
\\[1ex]
\xi_t&\vcentcolon=\Xi_t\cap\Sp^2
=\bigl\{\theta_1=t\pi,~\theta_3=0\bigr\}. 
\end{align*}
Note that $\xi_t$ is a meridian in $\Sp^2$
and $\Xi_t$ is the great sphere in $\Sp^3$ intersecting $\Sp^2$ orthogonally
along $\xi_t$. 
We will also have occasion to refer to the great sphere
\begin{align}\label{eqn:Z}
\rs &\vcentcolon=\{x\in\Sp^3\st x_3=0\}
\intertext{intersecting $\Sp^2$ orthogonally along $\Sp^1$, and to the following great circle which is disjoint from $\Sp^1$:}
\label{eqn:S1perp}
\Sp^1_\perp&\vcentcolon=\{x\in\Sp^3\st x_1=x_2=0\}.
\end{align}

\pgfmathsetmacro{\rc}{pi/4} 
\pgfmathsetmacro{\npar}{3}
\pgfmathsetmacro{\angle}{-125} 
\pgfmathsetmacro{\depth}{0.66} 
\pgfmathsetmacro{\cyl}{abs(\depth*cos(\angle)/2)} 
\pgfmathsetmacro{\tcrit}{atan(1/(\cyl*tan(\angle)))}
\pgfmathsetmacro{\offset}{0.33}
\begin{figure}\centering
\begin{tikzpicture}[line cap=round,line join=round,scale=2.25,semithick]
\path[yslant=tan(\angle)](0,-\rc)arc(-90:\tcrit:\cyl/1 and \rc)
coordinate(NW)arc(\tcrit:\tcrit+180:\cyl/1 and \rc)coordinate(SW);
\path(SW)--++(2*pi,0)coordinate(SE);
\path(NW)--++(2*pi,0)coordinate(NE);
\draw(pi/\npar/2,-pi/2)--++(\angle:-\depth)--++(0,pi)--++(\angle:\depth)--cycle;
\draw(2*pi,0)--++(\angle:-\depth)--++(-2*pi,0);
\shade[top color=FarbeB!50!black!70,bottom color=FarbeB!50!black!60,middle color=white,yslant=tan(\angle)]
(NW)arc(\tcrit:\tcrit-180:\cyl/1 and \rc)--(SE)
arc(\tcrit+180:\tcrit+360:\cyl/1 and \rc)--cycle
; 
\shade[draw=FarbeB,thick,top color=FarbeB!50!black!40,bottom color=FarbeB!50!black!40,middle color=white,yslant=tan(\angle)]
(NW)arc(\tcrit:\tcrit+180:\cyl/1 and \rc)--(SE)
arc(\tcrit-180:\tcrit-360:\cyl/1 and \rc)--cycle
; 
\fill[black!33](0,-pi/2)rectangle(2*pi,pi/2);
\begin{scope}[thick] 
\draw[->](0,0)--(2*pi+\offset,0)node[above]{~$\theta_1$};
\draw[->](0,-pi/2)--(0,pi/2+\offset)node[below right,inner sep=0]{~$\theta_2$};
\draw[->](0,0)--(\angle:\depth+0.1)node[below right,inner sep=0pt]{$\vphantom{\big\vert}\theta_3$};
\end{scope}
\draw[equator](0,0)--++(2*pi,0);
\pgfmathsetmacro{\klast}{\npar-1}
\foreach\k in {1,...,\klast}{
\draw[axis](\k*pi/\npar,-pi/2)--++(0,pi)node[anchor=base,shift={(-2ex,0)},pos=0.1]{$\xi_{\frac{\k}{n}}$};
\draw[axis](\k*pi/\npar+pi,-pi/2)--++(0,pi)node[anchor=base,shift={(-2ex,0)},pos=0.1]{$\xi_{\frac{\k}{n}}$};
}
\foreach\k in {0,1,2}{
\draw[axis](\k*pi,-pi/2)--++(0,pi)node[anchor=base,shift={(-2ex,0)},pos=0.1]{$\xi_0$};
}
\shade[top color=FarbeB!75!black!45,bottom color=FarbeB!75!black!33,middle color=white,yslant=tan(\angle)](0,\rc)arc(90:\tcrit+180:\cyl/1 and \rc)--(SE)arc(\tcrit-180:90-360:\cyl/1 and \rc)--cycle; 
\begin{scope}[FarbeB,thick]
\coordinate (E) at (2*pi,-\rc);
\coordinate (F) at (pi/\npar/2,-\rc);
\coordinate (G) at (pi/\npar/2, \rc);
\draw[yslant=tan(\angle),dotted](0,-\rc)arc(-90: 90:\cyl/1 and \rc);
\draw[yslant=tan(\angle)]       (0, \rc)arc( 90:270:\cyl/1 and \rc);
\draw[yslant=tan(\angle)]       (E)     arc(-90:270:\cyl/1 and \rc); 
\draw(SW)--(SE)node[above,midway](Clifford){Clifford torus $\T^2$};
\draw[dotted](NW)--(NE);
\draw[semithick,yslant=tan(\angle),dotted](F)arc(-90: 90:\cyl/1 and \rc);
\draw[semithick,yslant=tan(\angle)]       (G)arc( 90:270:\cyl/1 and \rc);
\draw[dashed,very thin](0,\rc)--++(2*pi,0)(0,-\rc)--++(2*pi,0);
\draw[dashed,very thin]
(\angle:\depth/2)--++(2*pi,0)
(\angle:-\depth/2)--++(2*pi,0)
;
\end{scope}
\foreach\k in {0,...,\klast}{
\draw[axis,dotted](\k*pi/\npar,-pi/2)--++(0,pi);
\draw[axis,dotted]({(\k+1)*pi/\npar+pi},-pi/2)--++(0,pi);
}
\draw[equator,dotted](0,0)--++(2*pi,0)node[midway,above,inner sep=1pt](equator){equator $\Sp^1$};
\draw(pi,-pi/2)node[below,inner sep=1pt]{south pole $\mathrlap{(0,0,-1,0)}$}
(pi, pi/2)node[above,inner sep=1pt]{north pole $\mathrlap{(0,0,1,0)}$};  
\draw[axis,dotted](pi,-pi/4)+(0,1.2ex)--(pi,0)(pi,1.2ex)--(pi,pi/2);
\draw(pi,pi/4+pi/8)node[black,fill=black!33,inner sep=1pt]{equatorial $\Sp^2$};
\draw(pi/\npar/2,-pi/2)--++(\angle:\depth)--++(0,pi)
--++(\angle:-\depth)node[below right,yslant=tan(\angle),yscale=1/\depth]{$\Xi_{\frac{1}{2n}}$}
;
\draw[dotted](pi/\npar/2,-pi/2)--++(\angle:-\depth)--++(0,pi)--++(\angle:\depth)--cycle;
\draw(\angle:\depth)--++(2*pi,0)node[above left,xslant=tan(-90-\angle),xscale=sqrt(1+tan(-90-\angle)^2)]{$Z$~}--++(\angle:-2*\depth);
\draw[dotted]
(\angle:\depth/2)--(\angle:-\depth)--++(2*pi,0)  
(pi/\npar/2,0)--++(\angle:-\depth) 
(pi/\npar/2,0)--++(\angle: \depth)
; 
\draw plot[hdash](0,pi/4)node[left]{$\frac{\pi}{4}$};
\draw plot[hdash](0,pi/2)node[left]{$\frac{\pi}{2}$};
\draw plot[plus](\angle:\depth/2)node[left]{$\frac{\pi}{4}$};
\draw plot[hdash](0,-pi/2)node[left]{$-\frac{\pi}{2}$};
\end{tikzpicture}
\caption{The submanifolds $\Sp^2,\Sp^1,\T^2,\xi_{\frac{i}{n}},\Xi_{\frac{1}{2n}},Z\subset\Sp^3$ in spherical coordinates $\theta_1,\theta_2,\theta_3$ for $n=\npar$.}%
\label{fig:symmetry}%
\end{figure}
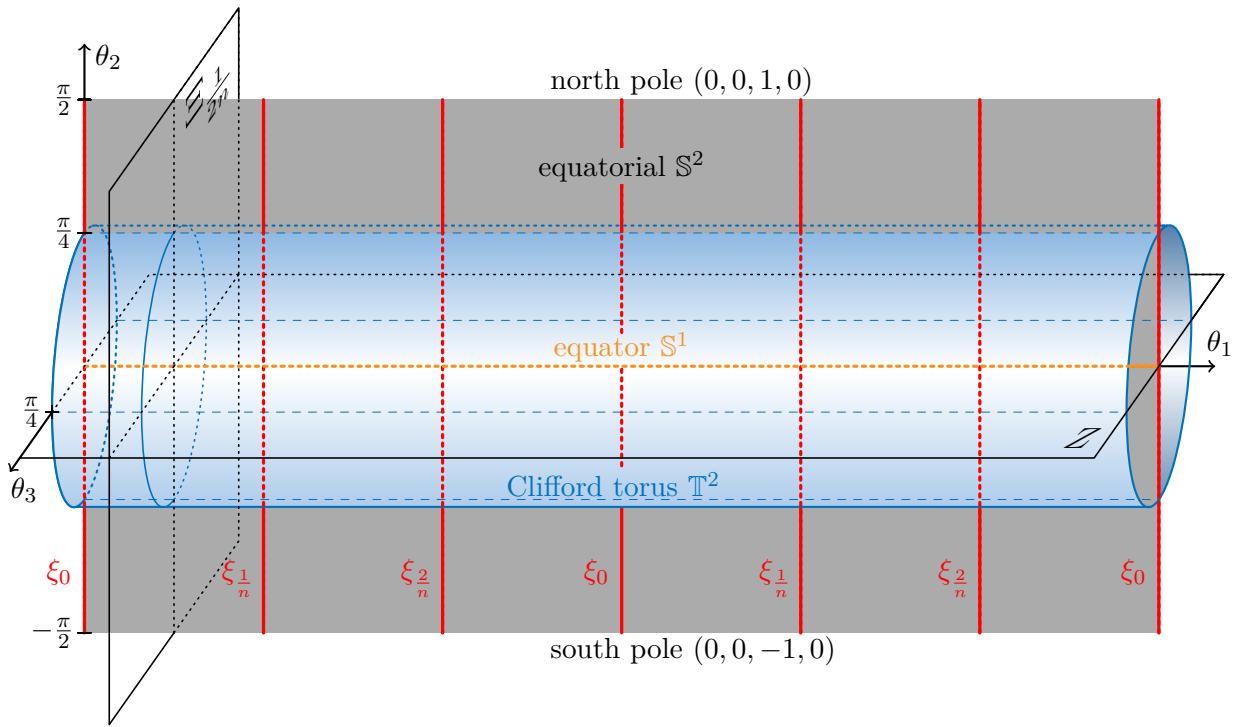
 
Given any subset $X\subset\Sp^3\subset\R^4$,
we define $\refl{X}\in\Ogroup(4)$ as the linear extension of
\begin{align*}
\refl{X}(p)=\begin{cases}
\hphantom{-}p & \text{ if $p\in\Span_{\R^4}(X)$, } \\	
-p & \text{ if $p\in\bigl(\Span_{\R^4}(X)\bigr)^\perp$. } \\	
\end{cases}
\end{align*} 
In practice $X$ will always be a great circle or a great sphere; for example 
\begin{align}\label{eqn:20250522} 
\underline{\Sp}^1(x)&=(x_1,x_2,-x_3,-x_4), &
\underline{\xi}_0(x)&=(x_1,-x_2,x_3,-x_4). 
\end{align}
Given a subset (or list of elements) $\omega$ in $\Ogroup(4)$ we understand $\sk{\omega}$ as the subgroup of $\Ogroup(4)$ which is generated by $\omega$. 
With this notation, we define the following concrete subgroups of isometries. 
\begin{align}\label{eqn:group}
\grp{A}_n&\vcentcolon=\sk{\refl{\xi}_0, ~ \refl{\Xi}_{\frac{1}{2n}}}, & 
\grp{G}_n&\vcentcolon=\sk{\grp{A}_n, ~ \refl{\Sp}^1}. 
\end{align}
The group $\grp{A}_n$ is isomorphic to the abstract dihedral group of order $4n$ and its action on $\Sp^3$ restricts to an action on the great sphere $\rs$. 
This restriction induces a monomorphism $\grp{A}_n \hookrightarrow \Ogroup(3)$, whose image is called the \emph{antiprismatic group} of order $4n$ (see~\cite{FranzKetoverSchulz,Carlottob,Carlotto}).

$\grp{A}_n$ consists of $n$ rotations along $\Sp^1$, the $n$ reflections $\refl{\Xi}_{(2i+1)/(2n)}$ for $i \in \Z$, the $n$ reflections (or half turns) $\refl{\xi}_{i/n}$ for $i \in \Z$, and $n$ transformations which are rotations along $\Sp^1$ and reflections on~$\Sp^1_\perp$.  
Every element of $\grp{A}_n$ preserves each (hemispherical) component of $\Sp^2 \setminus \Sp^1$
(and each of the two circles in $\Sp^2 \cap \T^2$),
while $\grp{G}_n$ contains transformations exchanging these two components (and circles).
We also define the cyclic subgroup
\begin{align}\label{eqn:Zn}
\Z_n\vcentcolon=\sk{\refl{\xi}_{\frac{1}{n}}\circ\refl{\xi}_0}
=\sk{\refl{\Xi}_{\frac{1}{2n}} \circ \refl{\Xi}_{-\frac{1}{2n}} } 
\end{align}
consisting of rotations along $\Sp^1$ (fixing $\Sp^1_\perp$ pointwise)
through multiples of angle $2\pi/n$. 

\begin{definition}\label{defn:equivariant}
Given a subgroup $G$ of $\Ogroup(4)$, a submanifold $\Sigma\subset\Sp^3$ is called \emph{$G$-equivariant} if the (natural) group action of $G$ on $\Sp^3$ restricts to an action on $\Sigma$ (which commutes with the inclusion map  $\Sigma\hookrightarrow\Sp^3$).  
\end{definition}

The following lemma is fundamental for the topological control in the proof of Theorem \ref{thm:main}. 

\begin{lemma}\label{lem:genus}
Given $2\leq n\in\N$, let $\Sigma\subset\Sp^3$ be any smooth, closed, connected, embedded, $\sk{\underline{\Sp}^1,\refl{\xi}_0,\refl{\xi}_{\frac{1}{n}}}$-equivariant (not necessarily minimal) surface of genus $g\leq2n$ which contains
the great circles $\Sp^1,\xi_0,\xi_{\frac{1}{n}}$.   
Then one of the following statements holds. 
\begin{enumerate}[label={\normalfont(\alph*)}]
\item\label{lem:genus-a} $\Sigma\cap\Sp^1_\perp$ contains exactly $2$ points and $g\in\{0,2n\}$. 
\item\label{lem:genus-b} $\Sigma\cap\Sp^1_\perp$ contains exactly $6$ points and $g=2n-2$.  
\item\label{lem:genus-c} $\Sigma\cap\Sp^1_\perp$ contains exactly $10$ points, $n=2$ and $g=4$.  
\end{enumerate}
\end{lemma}

\begin{proof}
We recall that $\sk{\underline{\Sp}^1,\Z_n}$ is a subgroup of $\sk{\underline{\Sp}^1,\refl{\xi}_0,\refl{\xi}_{\frac{1}{n}}}$. 
The singular locus of the $\Z_n$-action on $\Sp^3$ is the great circle $\Sp^1_\perp$ defined in \eqref{eqn:S1perp}, and thus disjoint from the singular locus of the $\underline{\Sp}^1$-action. 
By \cite[Lemma~3.4]{Ketover2016Equivariant} the surface $\Sigma$ either contains $\Sp^1_\perp$ or $\Sigma$ intersects $\Sp^1_\perp$ orthogonally in $j\in\N\cup\{0\}$ many points.  
The first case is ruled out by the assumption that $\Sigma$ is smooth containing 
$\xi_0\cup\xi_{\frac{1}{n}}$. 
Indeed, the great circles $\xi_0,\xi_{\frac{1}{n}}\subset\Sigma$ both intersect $\Sp^1_\perp$ orthogonally at the two points $(0,0,\pm1,0)$. 
In particular, $j\geq2$.  
Suppose that $j\geq3$.  
Then $(\Sigma\cap\Sp^1_\perp)\setminus\{(0,0,\pm1,0)\}\ni p$ is nonempty, i.\,e.~there exists some  $-1<a<1$ such that 
\begin{align*}
p&=\bigl(0,0,a,\sqrt{1-a^2}\bigr)\in \Sigma\cap\Sp^1_\perp. 
\intertext{%
Recalling the symmetries \eqref{eqn:20250522} which preserve $\Sigma$ and $\Sp^1_\perp$ as sets, we also have 
}
\underline{\Sp}^1(p)&=\bigl(0,0,-a,-\sqrt{1-a^2}\bigr)\in \Sigma\cap\Sp^1_\perp, \\
\underline{\xi}_0(p)&=\bigl(0,0,\hphantom{-}a,-\sqrt{1-a^2}\bigr)\in \Sigma\cap\Sp^1_\perp, \\
\underline{\Sp}^1\circ\underline{\xi}_0(p)&=\bigl(0,0,-a,\hphantom{-}\sqrt{1-a^2}\bigr)\in \Sigma\cap\Sp^1_\perp.
\end{align*}
We claim that $a\neq0$ and thus $j=2+4k$ for some nonnegative integer $k$. 
Indeed, if $p_0=(0,0,0,1)\in\Sigma$ then $\Sp^1_\perp$ intersects $\Sigma$ as well as the great circle $c\vcentcolon=\{x\in\Sp^3\st x_2=x_3=0\}$ orthogonally at $p_0$. 
Therefore, $c$ is tangent to $\Sigma$ in $p_0$ and since $\Sigma$ is $\underline{c}=\underline{\Sp}^1\circ\underline{\xi}_0$-equivariant, we have $c\subset\Sigma$ by \cite[Lemma~3.4\;(2)]{Ketover2016Equivariant}. 
However, the three great circles $c,\xi_0,\Sp^1\subset\Sigma$ intersect pairwise orthogonally at the point $(1,0,0,0)$ and this contradicts the smoothness of $\Sigma$.  

The quotient $\Sigma/\sk{\underline{\Sp}^1,\Z_n}$ has an orbifold structure (cf.~\cite[Prop.~13.2.1]{Thurston2022})
and by \cite[Prop.~13.3.1]{Thurston2022} the underlying space is a topological surface $\Sigma'$ with one boundary component arising from the fact that $\Sp^1\subset\Sigma$.  
Since the group $\sk{\underline{\Sp}^1,\Z_n}$ has order $2n$ 
and since $\Sigma/\sk{\underline{\Sp}^1,\Z_n}$ has $j/2$ orbifold points of order $n$, the general Riemann--Hurwitz formula (cf.~\cite[§\,13.3]{Thurston2022}) 
implies 
\begin{align}\label{eqn:Riemann-Hurwitz}
\chi(\Sigma)
&=2n\,\chi_{\text{orb}}\bigl(\Sigma/\sk{\underline{\Sp}^1,\Z_n}\bigr)
=2n\Bigl(\chi(\Sigma')-\tfrac{j}{2}(1-\tfrac{1}{n})\Bigr),
\end{align}
where $\chi_{\text{orb}}$ denotes the orbifold Euler number and $\chi(\Sigma)=2-2g$ and $\chi(\Sigma')=1-2g'$ for some $g'\in\N\cup\{0\}$ the Euler-Characteristics of the (topological) surfaces $\Sigma$ and $\Sigma'$.  
Recalling $j=2+4k$, equation \eqref{eqn:Riemann-Hurwitz} implies 
$2-2g=2n\bigl(1-2g'-(1+2k)(1-\tfrac{1}{n})\bigr)$, or equivalently 
\begin{align}\label{eqn:genus}
g&=2ng'+2k(n-1).
\end{align}
All variables in \eqref{eqn:genus} are nonnegative integers. 
Moreover, $n\geq2$ and $g\leq2n$ by assumption. 
Therefore, necessarily $g'\in\{0,1\}$ and $k\in\{0,1,2\}$.  
The proof concludes by distinguishing the following cases. 
\begin{itemize}[nosep]
\item If $k=0$ (and $g'\in\{0,1\}$) then $j=2$ and $g\in\{0,2n\}$ as in statement \ref{lem:genus-a}. 
\item If $k=1$ then $g'=0$, $j=6$ and $g=2n-2$ as in statement \ref{lem:genus-b}. 
\item If $k=2$ then $g'=0$, $j=10$ and $g=4n-4$, which violates the assumption $g\leq2n$ unless $n=2$ as in statement \ref{lem:genus-c}. 
\qedhere 
\end{itemize}
\end{proof}

We emphasise that the cases \ref{lem:genus-a}--\ref{lem:genus-c} in Lemma~\ref{lem:genus} are all topologically possible (see Figure~\ref{fig:wedge} for a visualization of the first two cases). 
Under the stronger assumption of $\grp{G}_n$-equivariance we analyze how the corresponding great spheres of symmetry may intersect the surface in question:

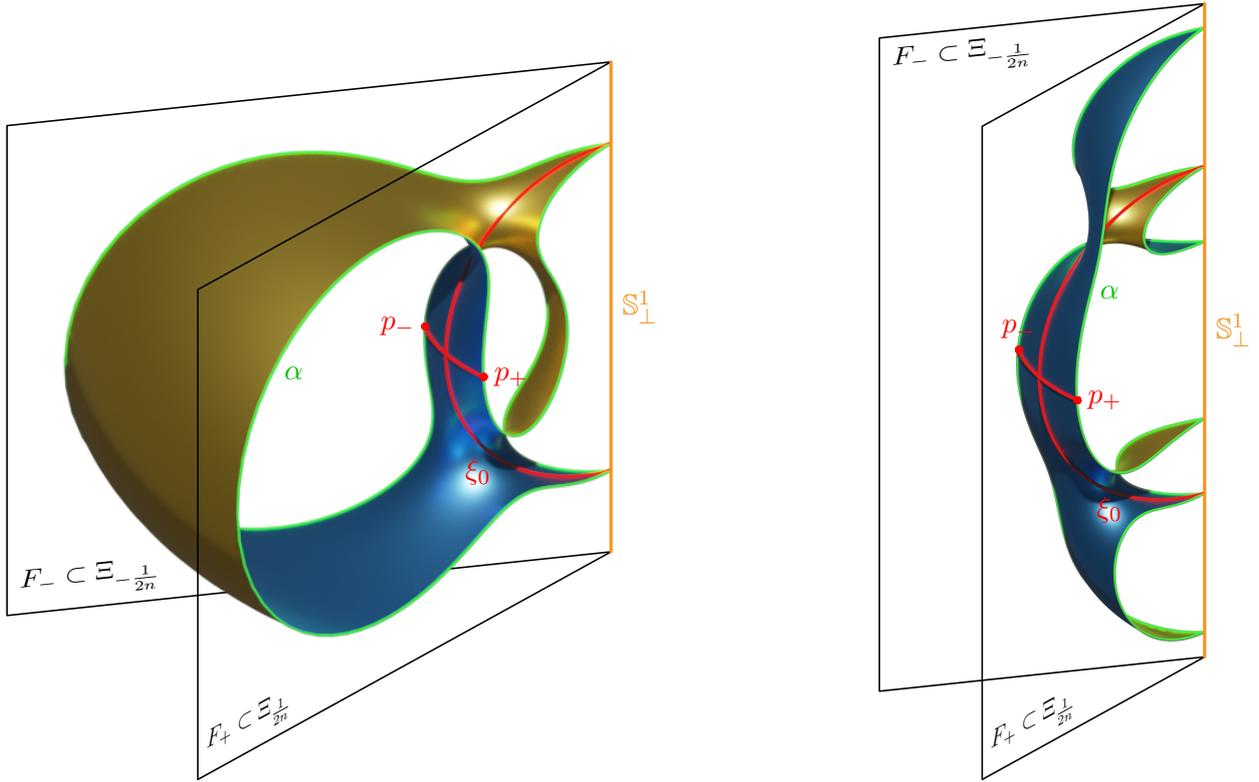
\begin{figure}[t]
\pgfmathsetmacro{\thetaO}{60}
\pgfmathsetmacro{\phiO}{90+60}
\tdplotsetmaincoords{\thetaO}{\phiO}
\pgfmathsetmacro{\globalscale}{2.5}
\pgfmathsetmacro{\surfacewidth}{\globalscale*6}
\pgfmathsetmacro{\npar}{5}
\pgfmathsetmacro{\angle}{90/\npar}
\begin{tikzpicture}[scale=\globalscale,tdplot_main_coords,line cap=round,line join=round,semithick]
\pgfmathsetmacro{\zmax}{1.5}
\pgfmathsetmacro{\xmax}{3.25}
\tdplotsetthetaplanecoords{-90/\npar}
\draw[tdplot_rotated_coords](-\zmax,0,0)--(\zmax,0,0)--++(0,\xmax,0)--++(-2*\zmax,0,0)coordinate(F-)--cycle;
\draw(0,0,0)node[inner sep=0]{\includegraphics[width=\surfacewidth cm]{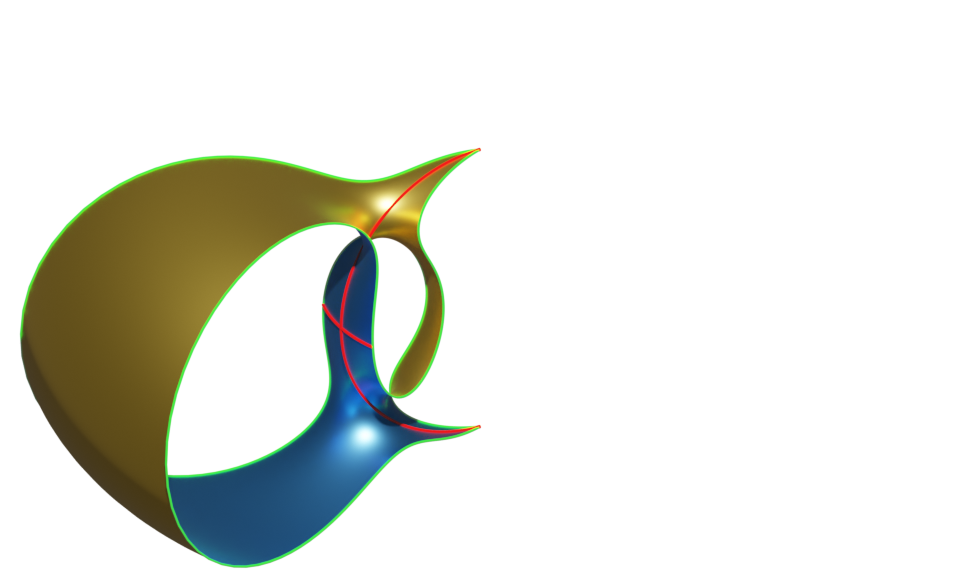}};
\pgfresetboundingbox
\path[tdplot_rotated_coords](-\zmax,0,0)--(\zmax,0,0)--++(0,\xmax,0)--++(-2*\zmax,0,0)coordinate(F-)--cycle;
\tdplotsetthetaplanecoords{\angle}
\draw[tdplot_rotated_coords](-\zmax,0,0)--(\zmax,0,0)--++(0,\xmax,0)--++(-2*\zmax,0,0)coordinate(F+)--cycle;
\draw[very thick,FarbeA](0,0,-\zmax)--(0,0,\zmax)node[right,midway]{$\Sp^1_\perp$};
\draw[red] plot[bullet]({cos(\angle)},{sin( \angle)},0)node[right]{$p_+$};
\draw[red] plot[bullet]({cos(\angle)},{sin(-\angle)},0)node[left]{$p_-$};
\draw[red,tdplot_screen_coords] (-0.7,-1)node[above]{$\xi_0$};
\draw[green!75!black,tdplot_screen_coords] (-1.67,-0.35)node[]{$\alpha$};
\tdplottransformmainscreen{-cos(\angle)}{-sin(\angle)}{0}
\pgfmathsetmacro{\VecHx}{\tdplotresx}
\pgfmathsetmacro{\VecHy}{\tdplotresy}
\tdplottransformmainscreen{0}{0}{1}
\pgfmathsetmacro{\VecVx}{\tdplotresx}
\pgfmathsetmacro{\VecVy}{\tdplotresy}
\node[cm={\VecHx,\VecHy,\VecVx,\VecVy,(0,0)},
inner sep=1ex,above right]at(F+){$F_+\subset\Xi_{\frac{1}{2n}}$};
\tdplottransformmainscreen{-cos(\angle)}{-sin(-\angle)}{0}
\pgfmathsetmacro{\VecHx}{\tdplotresx}
\pgfmathsetmacro{\VecHy}{\tdplotresy} 
\node[cm={\VecHx,\VecHy,\VecVx,\VecVy,(0,0)},
inner sep=1ex,above right]at(F-){$F_-\subset\Xi_{-\frac{1}{2n}}$}; 
\end{tikzpicture} 
\hfill
\pgfmathsetmacro{\surfacewidth}{\globalscale*2}
\begin{tikzpicture}[scale=\globalscale,tdplot_main_coords,line cap=round,line join=round,semithick]
\pgfmathsetmacro{\zmax}{2}
\pgfmathsetmacro{\xmax}{1.75}
\tdplotsetthetaplanecoords{-90/\npar}
\draw[tdplot_rotated_coords](-\zmax,0,0)--(\zmax,0,0)--++(0,\xmax,0)coordinate(F-)--++(-2*\zmax,0,0)--cycle;
\draw(0,0,0)node[inner sep=0]{\includegraphics[width=\surfacewidth cm]{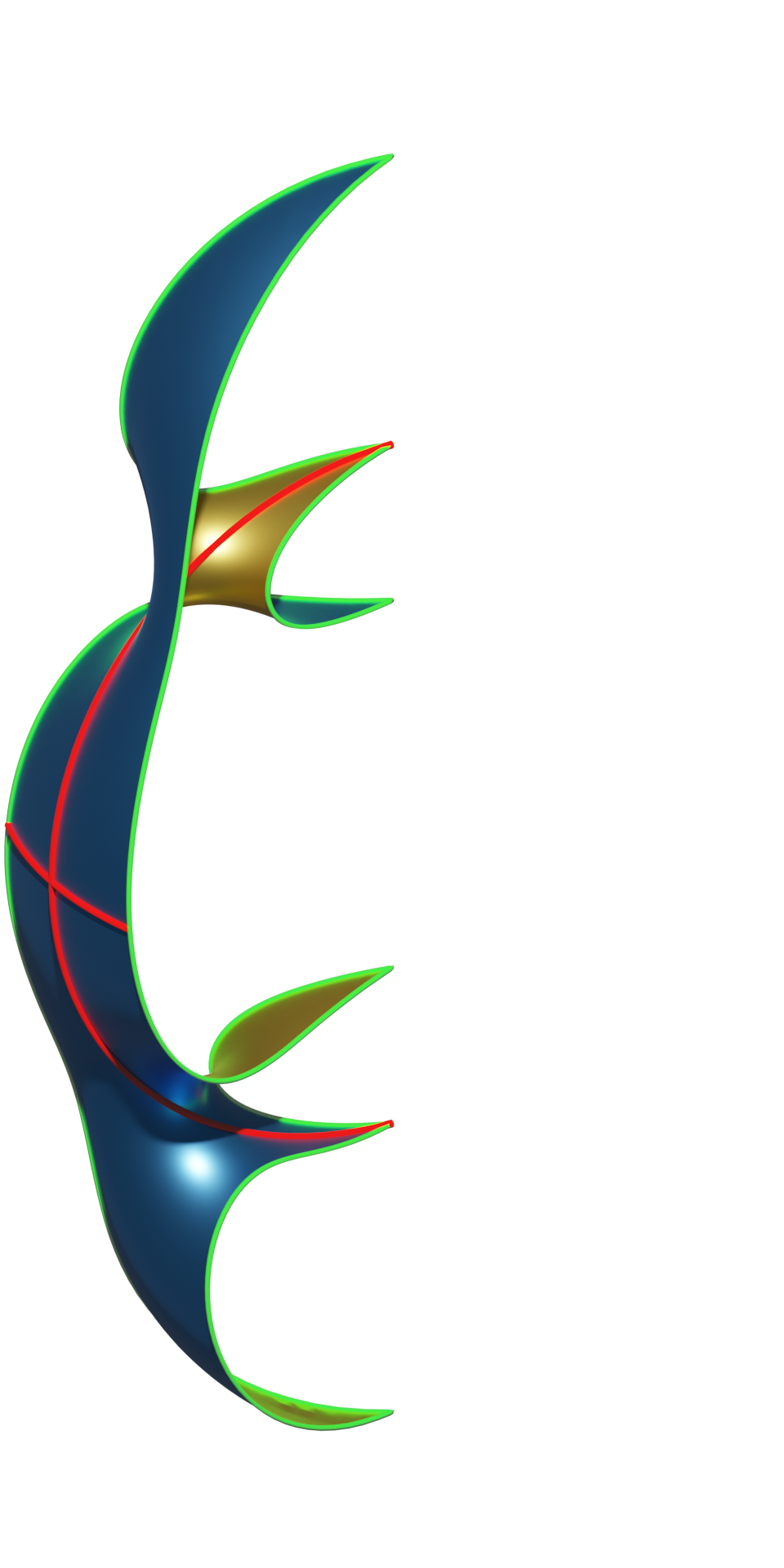}};
\pgfresetboundingbox
\path[tdplot_rotated_coords](-\zmax,0,0)--(\zmax,0,0)--++(0,\xmax,0)coordinate(F-)--++(-2*\zmax,0,0)--cycle;
\tdplotsetthetaplanecoords{\angle}
\draw[tdplot_rotated_coords](-\zmax,0,0)--(\zmax,0,0)--++(0,\xmax,0)--++(-2*\zmax,0,0)coordinate(F+)--cycle;
\draw[very thick,FarbeA](0,0,-\zmax)--(0,0,\zmax)node[right,midway]{$\Sp^1_\perp$};
\draw[red] plot[bullet]({cos(\angle)},{sin( \angle)},0)node[right]{$p_+$};
\draw[red] plot[bullet]({cos(\angle)},{sin(-\angle)},0)node[above]{$p_-$};
\draw[red,tdplot_screen_coords] (-0.5,-0.95)node[]{$\xi_0$};
\draw[green!75!black,tdplot_screen_coords] (-0.5,0.2)node[]{$\alpha$}; 
\tdplottransformmainscreen{-cos(\angle)}{-sin(\angle)}{0}
\pgfmathsetmacro{\VecHx}{\tdplotresx}
\pgfmathsetmacro{\VecHy}{\tdplotresy}
\tdplottransformmainscreen{0}{0}{1}
\pgfmathsetmacro{\VecVx}{\tdplotresx}
\pgfmathsetmacro{\VecVy}{\tdplotresy}
\node[cm={\VecHx,\VecHy,\VecVx,\VecVy,(0,0)},
inner sep=1ex,above right]at(F+){$F_+\subset\Xi_{\frac{1}{2n}}$};
\tdplottransformmainscreen{-cos(\angle)}{-sin(-\angle)}{0}
\pgfmathsetmacro{\VecHx}{\tdplotresx}
\pgfmathsetmacro{\VecHy}{\tdplotresy} 
\node[cm={\VecHx,\VecHy,\VecVx,\VecVy,(0,0)},
inner sep=1ex,below right]at(F-){$F_-\subset\Xi_{-\frac{1}{2n}}$};
\end{tikzpicture}
\caption{Left: 
Intersection of a $\grp{G}_n$-equivariant surface of genus $2n$ with the spherical lune $V\subset\Sp^3$. \\
Right: 
Intersection of a $\grp{G}_n$-equivariant surface of genus $2n-2$ with the spherical lune $V$. 
(The right image displays a \emph{nonminimal} surface in $\Sp^3$.)
}%
\label{fig:wedge}%
\end{figure}

\begin{corollary}
\label{cor:hemispheres}
Given $2 \leq n \in \N$ let
$\Sigma \subset \Sp^3$
be any smooth, closed, connected, embedded, 
$\grp{G}_n$-equivariant surface 
that contains the great circles $\xi_0$ and $\Sp^1$
and has genus $g\in\{1,\ldots,2n\}$.
Then the intersection of $\Sigma$
with the closure of any one of the $2n$ connected components of 
$\bigl(\bigcup_{i \in \Z} \Xi_{(2i+1)/(2n)}\bigr) \setminus \Sp^1_\perp$
(each a hemisphere)
is the union of $k$ pairwise disjoint embedded arcs $\alpha_1, \ldots, \alpha_k$,
each intersecting $\Sp^1_\perp$ orthogonally at two points,
and one of the following statements holds. 
\begin{enumerate}[label={\normalfont(\alph*)},nosep]
\item \label{lem:hemispheres:one} $k=1$  and $g=2n$,
\item \label{lem:hemispheres:three} $k=3$ and $g=2n-2$,
\item \label{lem:hemispheres:five} $k=5$,  $n=2$ and $g=4$.
\end{enumerate}
In every case,
for each $i \in \Z$
the union
$\bigcup_{j=1}^k \alpha_j \cup \refl{\xi}_{i/n}\bigcup_{j=1}^k \alpha_j$
is an embedded, connected, closed curve which is smooth at least away from $\Sp^1_\perp$.
\end{corollary}

\begin{proof}
We use arguments similar to ones used at the beginning of \cite[§\,4]{KW-lchar}.
Let $\lune$ be the closure of the connected component of
$
  \Sp^3 
  \setminus
  \bigcup_{i \in \Z}
    \Xi_{(2i+1)/(2n)}
$
containing $(1, 0, 0, 0)$;
it is a spherical wedge
(or dihedron or lune)
whose boundary $\partial\lune$
is the union of two hemispherical faces,
$F_+ \subset \Xi_{1/(2n)}$ and $F_- \subset \Xi_{-1/(2n)}$,
meeting at angle $\pi/n$ along $\Sp^1_\perp$.
Note that $\partial\lune$ is homeomorphic to $\Sp^2$
and $\refl{\Sp}^1(\lune)=\refl{\xi}_0(\lune)=\lune$.

Since $\Sp^1$ lies entirely on $\Sigma$
and has nonempty intersection with each face of $\lune$, 
we have $F_+ \cap \Sigma \neq \emptyset$
and $F_- \cap \Sigma \neq \emptyset$.
Since each face is contained in a sphere $\Xi$
of symmetry for $\Sigma$, it can intersect $\Sigma$ at any given point only orthogonally:  
In the (hypothetical) tangential case
$\Sigma$ is (using the embeddedness) locally a graph over $\Xi$,
so in fact, being $\refl{\Xi}$-equivariant, locally contained in $\Xi$.
Thus the set of points of tangential intersection
is both open and closed
(using the assumption that $\Sigma$ is smooth),
but $\Sigma$ is closed, connected,
and of strictly positive genus, so not a sphere. 

We conclude that $\Sigma \cap \partial\lune = \partial (\Sigma \cap \lune)$
is a union of pairwise disjoint embedded closed curves,
each a geodesic on $\Sigma$ except at points on $\Sp^1_\perp$,
which can only meet $\Sigma \cap \partial\lune$ orthogonally.
Furthermore, from the connectedness of $\Sigma$
and again the fact that each face of $\lune$
lies in a sphere of symmetry,
it follows that $\Sigma \cap \lune$ is also connected;
indeed the restriction
$\Pi|_\lune$ of the canonical projection
$
  \Pi
  \colon
  \Sp^3
  \to
  \Sp^3 / \sk{\refl{\Xi}_{1/(2n)}, \refl{\Xi}_{-1/(2n)}}
$
is a homeomorphism,
so we can apply $(\Pi|_\lune)^{-1} \circ \Pi$
to any path on $\Sigma$ to obtain one in $\Sigma \cap \lune$.

The equatorial great circle $\Sp^1$ intersects $F_+$ at a single point $p_+$ and $F_-$ at a single point $p_- = \refl{\xi}_0(p_+)$, as visualized in Figure \ref{fig:wedge}. 
Since $\Sp^1 \subset \Sigma$, we have $\{p_+, p_-\} \subset \partial (\Sigma \cap \lune)$.
Let $\alpha$ be the connected component of $\partial (\Sigma \cap \lune)$ through $p_+$. 
Then $\alpha$,
being an embedded closed piecewise smooth curve,
cuts the topological two-sphere $\partial\lune$ into two connected components.
Since $\refl{\Sp}^1$ is a symmetry of $\Sigma$ which fixes $p_+$,
it must also preserve $\alpha$
and exchange the two components of $(\partial\lune)\setminus\alpha$. 
On the other hand,
$\refl{\Sp}^1(p_-)=p_-$ and $\refl{\Sp}^1(\Sp^1_\perp) = \Sp^1_\perp$,
so we must have
$p_- \in \alpha$ and $\alpha \cap \Sp^1_\perp \neq \emptyset$.
Since $\refl{\xi}_0(p_-)= p_+$,
we also know that $\refl{\xi}_0(\alpha)=\alpha$.
In particular the set $\alpha \cap \Sp^1_\perp$
has cardinality at least two
and is a union of pairs of antipodal points.

Since the great spheres $\{\Xi_{(2i+1)/(2n)}\}_{i\in\Z}$ partition 
$\Sigma$ into $2n$ regions (each isometric to $V\cap\Sigma$), with the intersections having measure zero, the Gauss--Bonnet theorem applied to $\Sigma$ yields
\begin{equation*}
  2n \int_{\Sigma \cap \lune} K
  =
  \int_\Sigma K
  =
  2\pi(2-2g), 
\end{equation*}
where $K$ denotes the Gauss curvature of $\Sigma$. 
On the other hand,
since the faces of $\lune$ lie in spheres of symmetry for $\Sigma$,
the boundary $\partial(\Sigma \cap \lune)$ is a geodesic polygon in $\Sigma$,
with vertices on $\Sp^1_\perp$, each of angle $\pi/n$,
and the Gauss--Bonnet theorem applied to $\Sigma \cap \lune$ yields
\begin{equation*}
  \int_{\Sigma \cap \lune} K
    +j\Bigl(\pi-\frac{\pi}{n}\Bigr)
  =  2\pi(2-2\gamma-\beta),
\end{equation*}
where $\gamma$ and $\beta$ denote the genus and the number of boundary components of $\Sigma \cap \lune$ respectively, and where $j\in\{2,6,10\}$ is the cardinality of $\Sigma \cap \Sp^1_\perp$ as stated in Lemma~\ref{lem:genus}. 
Thus we obtain the relation
\begin{equation}
\label{eqn:GB}
  g=n(2\gamma + \beta - 2)
    + (n-1)\frac{j}{2}
    + 1.
\end{equation} 
Suppose $\beta \geq 2$.
Then there is a component $\alpha'$ of $\Sigma \cap \partial\lune$ lying in one component of $(\partial\lune) \setminus \alpha$.
Thus $\refl{\Sp}^1(\alpha')$ is a third distinct component of $\Sigma \cap \partial\lune$.
If $\alpha'$ meets $\Sp^1_\perp$,
it does so in at least two points,
as then does $\refl{\Sp}^1(\alpha')$,
and so we have $\beta \geq 3$ and $j \geq 6$,
but then \eqref{eqn:GB} implies
$g \geq 4n-2$, violating the assumed upper bound on the genus.
If instead $\alpha'$ does not meet $\Sp^1_\perp$,
then we have two additional distinct components $\refl{\xi}_0(\alpha')$ and $\refl{\xi}_0\circ\refl{\Sp}^1(\alpha')$ of $\Sigma \cap \partial\lune$,
so that $\beta \geq 5$ and $j \geq 2$,
in which case \eqref{eqn:GB} implies $g \geq 4n$, again a contradiction.
We conclude that $\beta = 1$.

Now consider the connected components $\alpha_1, \ldots, \alpha_k$ of $\alpha \cap F_+$, which are labelled such that $\alpha_1\ni p_+$.
As shown above, $\alpha_1$ intersects $\Sp^1_\perp$
in a pair of antipodal points. 
Moreover it cuts the disc $F_+$ into two connected components
which are exchanged under $\refl{\Sp}^1$. 
It follows that the intersection $\alpha \cap F_+$ has an odd number $k$
of connected components,
each of which has two endpoints on $\Sp^1_\perp$. 
In particular, $k=j/2$ and the claim about statements \ref{lem:hemispheres:one}, \ref{lem:hemispheres:three}, \ref{lem:hemispheres:five} follows directly from Lemma~\ref{lem:genus}. 
The final assertion is clear from the facts that $\beta = 1$,
$\refl{\xi}_0(\alpha)=\alpha$,
and that for any $i, \ell \in \Z$
the composite $\refl{\xi}_{i/n} \circ \refl{\xi}_{\ell/n}$
fixes $\Sp^1_\perp$ pointwise.
\end{proof}

\section{Construction of equivariant sweepouts}

This section establishes an effective $1$-parameter $\grp{G}_n$-sweepout of $\Sp^{3}$, meaning a family $\{\Sigma_t\}_{t\in[0,1]}$ of closed, embedded, $\grp{G}_n$-equivariant submanifolds of $\Sp^{3}$ such that $\Sigma_t$ is a smooth surface in $\Sp^{3}$ for every $t\in\interval{0,1}$, varying smoothly in $t\in\interval{0,1}$ and continuously, in the sense of varifolds, for $t\in[0,1]$ (see \cite[Definition 2.10]{FranzKetoverSchulz}).
Some of the ideas and the notation in the proof of the following lemma are inspired by \cite[Lemma 2.2]{Buzano2025}. 
Here and in the following sections, $\hsd^k$ denotes the $k$-dimensional Hausdorff measure on $\R^4$.  

\begin{lemma}\label{lem:sweepout}
For every $2\leq n\in\N$ there exists a $\grp{G}_n$-sweepout $\{\Sigma_t\}_{t\in[0,1]}$ of $\Sp^{3}$ with the following properties. 
\begin{enumerate}[label={\normalfont(\roman*)}]
\item\label{lem:sweepout-i} $\hsd^2(\Sigma_0)=\hsd^2(\Sigma_1)=\hsd^2(\Sp^{2})$ 
and $\hsd^2(\Sigma_t)
\leq2\pi^2+4\pi+\frac{1}{25}$ for every $t\in[0,1]$.
\item\label{lem:sweepout-ii} $\Sigma_t$ contains 
the great circle $\Sp^1$ and the meridians $\xi_0,\xi_{\frac{1}{n}},\ldots,\xi_{\frac{n-1}{n}}$ for every $t\in[0,1]$.  
\item\label{lem:sweepout-iii} $\Sigma_t$ has genus $2n$ for every $t\in\interval{0,1}$.
\end{enumerate}
\end{lemma}

\begin{proof}
Given $t\in\interval{0,1}$ the submanifold 
\begin{align}
C_t=\{x\in\Sp^3\st x_1^{2}+x_2^{2}=t^{2}\}=\{\cos(\theta_2)\cos(\theta_3)=t\}
\end{align} 
is the constant mean curvature torus in $\Sp^{3}$ at (spherical) distance $\arccos(t)$ from $\Sp^1$ and distance $\arcsin(t)$ from $\Sp^1_\perp$. 
In $\R^{4}$, the surface $C_t$ is the product of two (Euclidean) circles with radii 
$\sqrt{1-t^2}$ and $t$, and so has area $\hsd^2(C_t)=4\pi^2t\sqrt{1-t^2}$. 
For every $k\in\{0,\ldots,2n-1\}$ and every $0<t<1$ we consider the following subsets of $\Sp^{3}$, given in spherical coordinates (see Figure \ref{fig:subsets}).
\begin{align}
\label{eqn:subsets}
\begin{aligned}
\sigma_k&=\{k\pi/n\leq\theta_1\leq (k+1)\pi/n,~   0\leq\theta_2\leq\arccos(t),~\theta_3=0 \}\subset\Sp^{2}, \\
\zeta_k&=\{k\pi/n\leq\theta_1\leq (k+1)\pi/n,~    \arccos(t)\leq\theta_2\leq\pi/2,~\theta_3=0 \}\subset\Sp^{2}, \\
C_t'&=C_t\cap\{\theta_2\geq0,~\theta_3\geq0\}, \\
S_t'&=C_t'\cup  \bigl(\bigcup_{k\text{ even }}\sigma_k\bigr) \cup  \bigl(\bigcup_{k\text{ odd }}\zeta_k\bigr).
\end{aligned}
\end{align}
By construction,  
\begin{align}\label{eqn:upperarea}
\hsd^{2}(S_t')=\frac{1}{4}\hsd^{2}(C_t)+\frac{1}{4}\hsd^{2}(\Sp^{2})
=\pi^2t\sqrt{1-t^2}+\pi
\leq\frac{\pi^{2}}{2}+\pi
\end{align} 
for all $t\in\interval{0,1}$, 
and there exists an equivariant perturbation $\Sigma_t'$ of $S_t'$, depending smoothly on $0<t<1$ such that 
\begin{enumerate}[label={(\alph*)}]
\item\label{property:interior} The interior of $\Sigma_t'$ is smooth, embedded and contained in $\{0<\theta_3<\pi/2\}$.
\item\label{property:segments} $\Sigma_t'$ and $S_t'$ coincide along their (piecewise smooth) boundary. 
In particular $\Sigma_t'$ contains the meridian segment $\xi_{\frac{k}{n}}\cap\{0\leq\theta_2\leq\pi/2\}$ for every $k\in\Z$.
\item\label{property:equator} $\Sigma_t'$ and $S_t'$ meet tangentially along $\Sp^{1}$ and have the same limit as $t\to0$ respectively $t\to1$. 
\item\label{property:area} 
The slight relaxation $\hsd^{2}(\Sigma_t')<\frac{\pi^{2}}{2}+\pi+\frac{1}{100}$ of the upper area bound stated in \eqref{eqn:upperarea} holds.   
\end{enumerate}
Finally, we set 
\begin{align}\label{eqn:20250706}
\Sigma_t''&=\Sigma_t'\cup \refl{\xi}_0(\Sigma_t'), &
\Sigma_t&=\Sigma_t''\cup\refl{\Sp}^1(\Sigma_t'').
\end{align}
Properties \ref{property:interior}--\ref{property:equator} ensure that $\Sigma_t''\subset\Sp^{3}$ is indeed a smooth, embedded surface (with boundary) which can be 
``rotated'' around $\Sp^{1}$ to obtain the smooth, closed surface $\Sigma_t$.

\begin{figure}%
\pgfmathsetmacro{\tpar}{0.5}
\begin{tikzpicture}[line cap=round,line join=round,scale=2.25,semithick]
\fill[black!20](0,0)rectangle(2*pi,pi/2); 
\pgfmathsetmacro{\klast}{\npar-1}
\foreach\k in {0,1,...,\klast}{
\fill[black!40]
(2*\k*pi/\npar,0)rectangle++(pi/\npar,{rad(acos(\tpar))})
({(2*\k+1)*pi/\npar},pi/2)rectangle++(pi/\npar,{-pi/2+rad(acos(\tpar))})
;}
\foreach\k in {0,1,...,\klast}{
\draw[axis](\k*pi/\npar,0)--++(0,pi/2)node[above left, pos=0.05]{$\xi_\k$}; 
\draw[axis](\k*pi/\npar+pi,0)--++(0,pi/2)node[above left, pos=0.05]{$\xi_\k$}; 
}
\draw[thick,->](0,0)--(2*pi+\offset/2,0)node[above=1ex,inner sep=0]{$\theta_1$};
\draw[thick,->](0,0)--(0,pi/2+\offset)node[below right,inner sep=0]{~$\theta_2$}; 
\pgfmathsetmacro{\klast}{2*\npar-1}
\foreach\k in {0,1,...,\klast}{
\path(\k*pi/\npar,0)--++(pi/\npar,{rad(acos(\tpar))})node[midway]{$\sigma_\k$}; 
\path(\k*pi/\npar,pi/2)--++(pi/\npar,{-pi/2+rad(acos(\tpar))})node[midway]{$\zeta_\k$}; 
}
\draw[dashed,FarbeB](0,{rad(acos(\tpar))})--++(2*pi,0);
\draw[axis](2*pi,0)--++(0,pi/2)node[above left, pos=0.05]{$\xi_0$};
\draw plot[hdash](0,pi/2)node[left]{$\frac{\pi}{2}$};
\draw plot[hdash](0,{rad(acos(\tpar))})node[left]{$\arccos(t)$};
\draw plot[plus](0,0)node[below]{$0$};
\draw plot[vdash](pi,0)node[below]{$\pi$};
\draw plot[vdash](2*pi,0)node[below]{$2\pi$};
\end{tikzpicture}
\caption{The subsets $\sigma_k,\zeta_k\subset\Sp^{2}\subset\Sp^{3}$ defined in \eqref{eqn:subsets} using spherical coordinates $\theta_1,\theta_2,\theta_3$.}%
\label{fig:subsets}%
\end{figure}
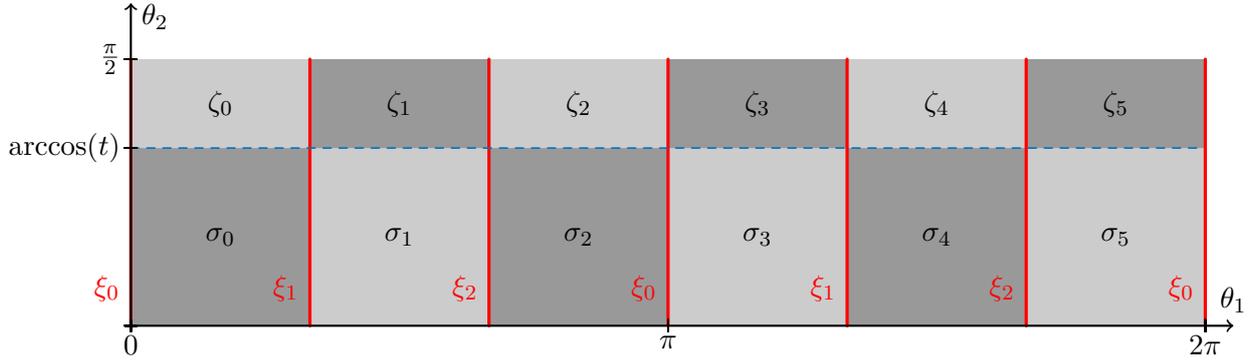

By construction, $\Sigma_t$ depends smoothly on $t\in\interval{0,1}$.  
Setting $\Sigma_0\vcentcolon=\Sp^{2}=\vcentcolon\Sigma_1$ we obtain continuity in the sense of varifolds for all $t\in[0,1]$. 
By construction, $\Sigma_t$ is $\grp{G}_n$-equivariant containing $\Sp^1$ and the $n$ meridians as stated in \ref{lem:sweepout-ii}. 
Statement \ref{lem:sweepout-i} about the area of $\Sigma_t$ follows from properties \ref{property:equator} and \ref{property:area}. 

For any $0<t<1$ the surface $\Sigma_t$ intersects the great circle $\Sp_\perp^1$ by construction only at the two points $(0,0,\pm1,0)\in\xi_0\cap\xi_{\frac{1}{n}}$. 
Since $\Sigma_t$ is clearly not a sphere, Lemma~\ref{lem:genus}\,\ref{lem:genus-a} implies that $\Sigma_t$ has genus $2n$ as claimed in \ref{lem:sweepout-iii}.  
\end{proof}

\begin{remark} 
Consider the group $\grp{H}_n\vcentcolon=\sk{\grp{A}_n, ~ \refl{Z}}$ in place of $\grp{G}_n=\sk{\grp{A}_n, ~ \refl{\Sp}^1}$. 
The same construction as for Lemma~\ref{lem:sweepout} yields an 
$\grp{H}_n$-sweepout of $\Sp^3$ if we replace the second assertion in \eqref{eqn:20250706} by 
$\Sigma_t=\Sigma_t''\cup\refl{Z}(\Sigma_t'')$. 
An $\grp{H}_n$-equivariant min-max procedure using this $\grp{H}_n$-sweepout yields the existence of an $\grp{H}_n$-equivariant minimal surface $L_n\subset\Sp^3$ with the same properties as stated in Theorem~\ref{thm:main}\,\ref{thm:main-genus}--\ref{thm:main-area}. 
However with the group $\grp{H}_n$ in place of $\grp{G}_n$ it is much harder to prove that the resulting minimal surface is actually distinct from any of the Lawson surfaces. 
Indeed, given $n\in\N$ the Lawson surface $\xi_{2,n-1}$ is $\grp{H}_n$-equivariant with genus $2n-2$ and thus compatible with a possible limit in the $\grp{H}_n$-equivariant min-max procedure. 
Numerical simulations suggest that $L_2$ is indeed congruent to $\xi_{2,1}$ and that $L_3$ is congruent to the Karcher--Pinkall--Sterling \cite{KarcherPinkallSterling} surface of genus $6$ (unlike the surfaces constructed in Theorem~\ref{thm:main}).  
In general we would expect $L_n$ to have slightly less area than $\Gamma_n$ as a consequence of the additional even symmetry $\refl{Z}$ in place of the odd symmetry $\refl\Sp^1$.   
\end{remark}

\section{Width estimate}

Given any fixed integer $2\leq n\in\N$ let $\{\Sigma_t\}_{t\in[0,1]}$ be the $\grp{G}_n$-sweepout constructed in Lemma \ref{lem:sweepout}. 
Its $\grp{G}_n$-saturation $\Pi$ is understood as the set of all $\{\Phi(t,\Sigma_t)\}_{t\in[0,1]}$, where $\Phi\colon[0,1]\times\Sp^{3}\to\Sp^{3}$ is a smooth map such that $\Phi(t,\cdot)$ is a diffeomorphism of $\Sp^{3}$ which commutes with the $\grp{G}_n$-action for all $t\in[0,1]$ and coincides with the identity for $t\in\{0,1\}$.  
The number  
\[
W_\Pi\vcentcolon=\adjustlimits\inf_{\{\Lambda_t\}\in \Pi~}\sup_{t\in[0,1]}\hsd^2({\Lambda_t})
\]
is called \emph{min-max width} of $\Pi$ (cf.~\cite[Definition 2.12]{FranzKetoverSchulz}). 
The min-max approach (cf.~\cite[Theorem~1.4]{FranzSchulz} and references therein) requires a width estimate of the form $W_\Pi>\max\{\hsd^2(\Sigma_0),\hsd^2(\Sigma_1)\}$. 
Such an estimate is typically proven by applying the isoperimetric inequality. 
In the case of the sphere $\Sp^3$, the isoperimetric inequality implies that any Borel set $F\subset\Sp^3$ of measure $\hsd^3(F)=\hsd^3(\Sp^3)/2$ has perimeter 
$P(F)\geq4\pi$ with equality if and only if $E$ is a hemisphere.  
Quantitative version of this statement are available in the literature (see e.\,g.~\cite[(5.38)]{Fusco2015} and \cite{Fusco2017}). 

For our purposes we also need to take the prescribed symmetry group $\grp{G}_n$ into account. 
In the context of sets with finite perimeter, we relax Definition \ref{defn:equivariant} by saying that a subset $F\subset\Sp^3$ is \emph{$\grp{G}_n$-equivariant}, if for all $\varphi\in\grp{G}_n$, the set $\varphi(F)$ coincides either with $F$ or with $\Sp^3\setminus F$ up to a negligible set (cf.~\cite[Definition 3.3]{Carlotto2022}).
The proof of the following lemma is analogous to the proof of \cite[Lemma~4.7]{Carlotto2022} (see also \cite[Lemma~2.3]{Buzano2025}). 

\begin{lemma}[Stability of the isoperimetric inequality in $\Sp^{3}$]
\label{lem:isoperimetric}
For every $\varepsilon>0$ there exists $\delta>0$ such that given a $\grp{G}_n$-equivariant Borel set $F\subset\Sp^3$ with Lebesgue measure $\hsd^3(F)=\hsd^3(\Sp^3)/2$ and perimeter $P(F)\leq4\pi+\delta$, there exists a $\grp{G}_n$-equivariant hemisphere $\tilde F$ with $\hsd^3(F\symdiff\tilde{F})\leq\varepsilon$.
\end{lemma}

\begin{lemma}[width estimate]\label{lem:width}
$\hsd^{2}(\Sp^2)<W_\Pi\leq2\pi^2+4\pi+\frac{1}{25}<3\hsd^{2}(\Sp^2)$. 
\end{lemma}

\begin{proof}
By construction, $\Sigma_t\subset\Sp^{3}$ divides $\Sp^{3}$ into two domains of equal volume for any $t\in[0,1]$.  
Given $0\leq t<1$, let $E_t\subset\Sp^{3}$ be the connected component of $\Sp^{3}\setminus\Sigma_t$ containing the point $e_4=(0,0,0,1)$. 
In particular, $E_0=\{\theta_3>0\}$ in spherical coordinates. 
Setting $E_1=\{\theta_3<0\}$ we obtain a $1$-parameter family 
$\{E_t\}_{t\in[0,1]}$ of subsets with boundary $\partial E_t=\Sigma_t$, being continuous in the sense that for any $t_0\in[0,1]$
\begin{align}\label{eqn:continuity}
\lim_{t\to t_0}\hsd^{3}(E_t\symdiff E_{t_0})&=0.
\end{align}
Let $\{\Lambda_t\}_{t\in[0,1]}\in\Pi$ be arbitrary. 
Then there exists a smooth map $\Phi\colon[0,1]\times\Sp^{3}\to\Sp^{3}$, where $\Phi(t,\cdot)$ is a $\grp{G}_n$-equivariant diffeomorphism for all $t\in[0,1]$ which coincides with the identity for $t\in\{0,1\}$, such that $\Lambda_t=\Phi(t,\Sigma_t)$. 
Let $F_t\vcentcolon=\Phi(t,E_t)$. 
Then the $1$-parameter family $\{F_t\}_{t\in[0,1]}$ is continuous in the sense of \eqref{eqn:continuity} with $F_0=\{\theta_3>0\}$ and $F_1=\{\theta_3<0\}$. 
Given $\varepsilon>0$ let $\delta>0$ be as in Lemma~\ref{lem:isoperimetric} and suppose that 
\[\sup_{t\in[0,1]}\hsd^{2}(\Lambda_t)\leq4\pi+\delta.\]

Then, for every $t\in[0,1]$, there exists a $\grp{G}_n$-equivariant hemisphere $\tilde F_t$ such that $\hsd^{3}(F_t\symdiff\tilde F_t)\leq\varepsilon$. 
By the triangle inequality for symmetric differences we also have 
\begin{align}\label{eqn:20250422}   
\hsd^{3}(\tilde F_s\symdiff\tilde F_t)\leq\hsd^{3}(F_s\symdiff F_t)+2\varepsilon
\end{align}
for any $s,t\in[0,1]$.
By Lemma~\ref{lem:greatspheres} stated below there are only finitely many candidates for $\tilde F_t$. 
In particular, since the family $\{F_t\}_{t\in[0,1]}$ is continuous, \eqref{eqn:20250422} implies that $\tilde F_t$ is constant in $t$ (provided that $\varepsilon$ has been chosen sufficiently small), but this contradicts the fact that $F_0$ is the complement of $F_1$ in $\Sp^{3}$. 

The upper bound is a direct consequence of Lemma~\ref{lem:sweepout}\,\ref{lem:sweepout-i} and the definition of min-max width.
\end{proof}

\begin{lemma}[$\grp{A}_n$-equivariant great spheres]
\label{lem:greatspheres}
For any $2\leq n\in\N$, the only $\grp{A}_n$-equivariant great spheres in $\Sp^3$ are $\Sp^2$ and $\rs$.
\end{lemma}

\begin{proof}
By the centers of a given great sphere $S\subset\Sp^3$
we mean the unique pair of antipodal points 
at distance $\pi/2$ from $S$ in $\Sp^3$. 
Let $S\subset\Sp^3$ be any $\grp{A}_n$-equivariant great sphere
with centers $\pm p\in\Sp^3$. 
For $S$ to be equivariant with respect to the reflections
through all great spheres $\{\Xi_{(2k+1)/(2n)}\}_{k\in\Z}$,
its centers $\pm p$ must lie in their common intersection,
which is the great circle $\Sp^1_\perp$. Indeed, the only other possibility, that $S=\Xi_{(2i+1)/(2n)}$ for some $i \in \Z$, is ruled out by the condition $\refl{\xi}_{i/n}(S)=S$.

The fact that $\refl{\xi}_{0}\in\grp{A}_n$ also implies that the centers $\pm p$ of $S$ lie either on the great circle $\xi_0$ or on  $\xi_0^\perp\vcentcolon=\{x \in \Sp^3 \st x_1=x_3=0\}$. 
We conclude that either $\pm p\in\Sp^1_\perp\cap\xi_0=\{(0,0,\pm1,0)\}$, in which case $S=\rs$, or $\pm p\in\Sp^1_\perp\cap\xi_0^\perp=\{(0,0,0,\pm1)\}$, in which case $S=\Sp^2$.
\end{proof}

\section{Topological and geometric control}
\label{sec:topology}

This section contains the proof of the main existence result Theorem~\ref{thm:main}, which is split into three parts:
Section \ref{subsec:genus} focuses on controlling the topology, 
Section \ref{subsec:symmetry} determines the full symmetry group and 
in Section \ref{subsec:index} we estimate the Morse index from below. 
Finally, in Section \ref{subsec:distinctness}, we distinguish the constructed surfaces from any of the previously known minimal surfaces in $\Sp^3$. 

\subsection{Genus and area}
\label{subsec:genus}

\begin{proof}[Proof of Theorem~\ref{thm:main}\,\ref{thm:main-genus}--\ref{thm:main-axes}]
Given $2\leq n\in\N$ let $\{\Sigma_t\}_{t\in[0,1]}$ be the $\grp{G}_n$-sweepout of $\Sp^3$ constructed in Lemma \ref{lem:sweepout}. 
The width estimate stated in Lemma \ref{lem:width} implies that the equivariant min-max theorem   
\cite[Theorem~2.14, see Remark~1.2]{FranzKetoverSchulz} applies. 
(Indeed, using the notation of \cite[Table~2]{FranzKetoverSchulz}, the singular locus of the $\grp{G}_n$-action consists only of great circles and great spheres of types $\mathord*11$, $\mathord*22$ and $\mathord*nn$.)
Hence, there exists a min-max sequence converging in the sense of varifolds to $m\Gamma_n$, 
where $\Gamma_n$ is a connected, embedded, $\grp{G}_n$-equivariant minimal surface in $\Sp^3$ and where the multiplicity $m$ is a positive integer.  
Here we rely on the fact that a closed minimal surface in $\Sp^3$ is necessarily connected because $\Sp^3$ being complete and connected with positive Ricci curvature has the Frankel \cite{Frankel1966} property.  
Moreover, $\Gamma_n$ is orientable being a closed, embedded surface in $\Sp^3$.     
Recalling Lemma~\ref{lem:width}, we have 
\begin{align}\label{eqn:20250612}
m\hsd^2(\Gamma_n)
&=W_\Pi\in\interval{4\pi,2\pi^2+4\pi+\tfrac{1}{25}}\subset\interval{4\pi,12\pi}.
\end{align}
The upper area bound in claim \ref{thm:main-area} follows from \eqref{eqn:20250612} and we also obtain $m\in\{1,2\}$ since necessarily $\hsd^2(\Gamma_n)\geq4\pi$ (see e.\,g.~\cite{Li1982}). 
All the surfaces in the min-max sequence satisfy property \ref{thm:main-axes} stated in Theorem~\ref{thm:main}. 
Appealing to the $\grp{G}_n$-equivariance, we conclude that $\Gamma_n$ also satisfies property \ref{thm:main-axes} and that the multiplicity $m$ must be odd (cf.~\cite[Theorem~3.2\,(iii)]{Ketover2016FBMS}). 
Consequently, $m=1$.

By Lemma~\ref{lem:sweepout}\,\ref{lem:sweepout-iii}, all the surfaces in the min-max sequence have genus $2n$. 
The topological lower semicontinuity result stated in \cite[Theorem~2.14]{FranzKetoverSchulz} 
(cf.~\cite{Ketover2019} and \cite[Theorem~1.8]{FranzSchulz}) 
implies that the genus of $\Gamma_n$ is at most $2n$. 
Since $\hsd^2(\Gamma_n)>\hsd^2(\Sp^2)$ by \eqref{eqn:20250612} (using that $m=1$), we conclude $\genus(\Gamma_n)\geq1$ from Almgren's \cite[Lemma~1]{Almgren1966} uniqueness result. 
Lemma~\ref{lem:genus} then implies that the genus of $\Gamma_n$ is either $2n$ or $2n-2$ which proves claim \ref{thm:main-genus}.
In particular $\Gamma_n$ has genus at least $2$, so, by \cite[Theorem~B]{MarquesNeves2014}, area greater than $2\pi^2$. 
\end{proof}

\subsection{Full symmetry group}
\label{subsec:symmetry}

By construction the minimal surface $\Gamma_n\subset\Sp^3$ is $\grp{G}_n$-equivariant for any $2\leq n\in\N$. 
In the following we prove that there are no extra symmetries if $n\geq4$. 

\begin{definition}
The \emph{symmetry group} of an embedded surface $\Sigma \subset \Sp^3$ is
\[
\symgrp{\Sigma}\vcentcolon=\{\Phi \in \Ogroup(4) \st \Phi(\Sigma) = \Sigma\}.
\]
\end{definition}

\begin{lemma}[Reflections through circles orthogonally intersecting $\Sp^1$]
\label{lem:refls_through_ortho_circles}
Let $2 \leq n \in \N$, and let $\xi\subset\Sp^3$ be any great circle intersecting $\Sp^1$ orthogonally.
If $\refl{\xi} \in \symgrp{\Gamma_n}$, then $\refl{\xi} \in \grp{G}_n$ 
and there exists $i\in\Z$ such that
either $\xi=\xi_{i/n}$
or $\xi$ orthogonally intersects $\xi_{i/n}$ at two points on $\Sp^1$.
In particular, if $\xi\subset\Gamma_n$, then $\xi \in \{\xi_{i/n}\}_{i \in \Z}$.
\end{lemma}

\begin{proof}
We consider the various possible locations for the antipodal points $\pm p\in\Sp^3$ where $\xi$ intersects $\Sp^1$ orthogonally.
If $\xi \cap \Sp^1 = \xi_{i/n} \cap \Sp^1$ for some $i \in \Z$,
then $\xi$ either coincides with $\xi_{i/n}$ or intersects it orthogonally because $\Gamma_n$ is embedded containing both $\Sp^1$ and $\xi_{i/n}$. 
In both cases we obtain $\refl{\xi}\in\grp{G}_n$.

If $\xi\subset\Xi_{(2i+1)/(2n)}$ for some $i\in\Z$,
then 
the composite $\refl{\xi}\circ\refl\Xi_{(2i+1)/(2n)}$ 
is reflection through the great sphere $\Xi'$ orthogonal to $\Xi_{(2i+1)/(2n)}$
and containing $\xi$.    
Since $\xi$ intersects $\Sp^1$ by assumption, we also have $\Sp^1\subset\Xi'$, 
but $\refl{\Sp}^1\circ \refl{\Xi}'$ is reflection through
the great sphere $\Xi''$ orthogonally intersecting $\Xi$ along $\Sp^1$.
Thus $\Gamma_n$ must be tangential along $\Sp^1$
to either $\Xi'$ or $\Xi''$
and therefore actually equal to one of them
(as explained for example in the second paragraph
of the proof of Corollary \ref{cor:hemispheres}),
in contradiction to the fact that $\Gamma_n$ has genus at least two.

Thus we can assume that the points $\pm p$ do not lie on $\xi_{i/n}$ or $\Xi_{(2i+1)/(2n)}$ for any $i\in\Z$. 
In this case there is some $3\leq k \in \N$
for which $\Gamma_n$ is $\grp{G}_{kn}$-equivariant,
which violates Lemma \ref{lem:genus},
given that $\Gamma_n$ has genus
at most $2n$, so strictly less than $2kn-2$.

Finally, if $\xi \subset \Gamma_n$, then $\refl{\xi} \in \symgrp{\Gamma_n}$,
but since $\Gamma_n$ also contains $\Sp^1$ and $\bigcup_{i \in \Z} \xi_{i/n}$,
the embeddedness of $\Gamma_n$ concludes the proof.
\end{proof}

\begin{lemma}[Symmetries preserving $\Sp^1$]
\label{lem:syms_preserving_Sp1}
Let $2 \leq n \in \N$ and $\Phi \in \symgrp{\Gamma_n}$.
If $\Phi(\Sp^1)=\Sp^1$ then $\Phi\in\grp{G}_n$.
\end{lemma}

\begin{proof}
Since $\Phi$ preserves $\Sp^1$,
it also preserves $\Sp^1_\perp$,
and its restriction to each of these two great circles is either a reflection or rotation.
If both restrictions are reflections,
then $\Phi$ is a reflection through a great circle orthogonally intersecting $\Sp^1$ (and $\Sp^1_\perp$),
so we conclude by appealing to Lemma \ref{lem:refls_through_ortho_circles}.
In the remaining cases, we reduce to this last one as follows.
If both restrictions are rotations,
then we compose with $\refl{\xi}_0$ to reduce to the previous case.
If $\Phi|_{\Sp^1}$ is a reflection and $\Phi|_{\Sp^1_\perp}$ a rotation,
then composition with $\refl{\Xi}_{1/(2n)}$ again reduces to the first case.
Finally, if $\Phi|_{\Sp^1}$ is a rotation and $\Phi|_{\Sp^1_\perp}$ a reflection,
then $\refl{\xi}_0 \circ \refl{\Xi}_{1/(2n)} \circ \Phi$
is a reflection through a great circle orthogonally intersecting $\Sp^1$.
\end{proof}

The proof of the Lemma \ref{lem:invariance-of-Sp1} below proceeds in part by counting umbilics on $\Gamma_n$.
At an umbilic point a minimal surface in $\Sp^3$ agrees with its totally geodesic tangent sphere to at least second order. 
We define the \emph{multiplicity of an umbilic point} to be this order of contact less one.

\begin{lemma}[Umbilics]
\label{lem:umbilics}
Let $2 \leq n \in \N$.
\begin{enumerate}[label={\normalfont(\roman*)},nosep]
\item \label{lem:umbilics:count}  Counting with multiplicity, $\Gamma_n$ has at most $8n-4$ umbilics.
\item \label{lem:umbilics:obvious} Assuming $n\geq 3$, the two points $(0,0,\pm 1,0)$ are umbilics of multiplicity at least $n-2$.
\end{enumerate}
\end{lemma}

\begin{proof}
Claim \ref{lem:umbilics:count} follows from \cite[Proposition~1.5]{Lawson1970} and Theorem~\ref{thm:main}\,\ref{thm:main-genus}.
Claim \ref{lem:umbilics:obvious} is evident from the fact that the $n$ great circles $\xi_{i/n}$ with $i \in \Z$ lie on $\Gamma_n$ and all intersect at each of the two points in question.
\end{proof}

\begin{lemma}[Invariance of $\Sp^1$ under a symmetry]
\label{lem:invariance-of-Sp1}
Let $4 \leq n \in \N$ and $\Phi \in \symgrp{\Gamma_n}$.
Then $\Phi(\Sp^1)=\Sp^1$.
\end{lemma}

\begin{proof}
We will prove the contrapositive, assuming that $\Phi \in \symgrp{\Gamma_n}$ maps $\Sp^1$ to a great circle $\Phi(\Sp^1) \neq \Sp^1$ and concluding that $n<4$.
By Lemma~\ref{lem:umbilics}\,\ref{lem:umbilics:obvious} the two poles $\pm p\vcentcolon = (0,0,\pm1,0)$ are umbilic points of multiplicity at least $n-2$ on $\Gamma_n$.  
We distinguish the following three cases. 
\begin{enumerate}[label={\itshape Case \arabic*:},wide,ref={\arabic*}]
\item\label{case1} $\Phi(\Sp^1)$ intersects $\Sp^1$ orthogonally. 
By Lemma \ref{lem:refls_through_ortho_circles}
we have $\Phi(\Sp^1) = \xi_{i/n}$
for some $i \in \Z$.
Each of the two umbilics $\pm p$
lies on every $\xi_{i/n}$ with $i \in \Z$,
and so in particular on $\Phi(\Sp^1)$.
Applying $\Phi^{-1}$
we find that two additional umbilics, each of multiplicity at least $n-2$ lie on $\Sp^1$.
From the $\grp{G}_n$-equivariance
we then in fact have $2n$ umbilics of multiplicity $n-2$
on $\Sp^1$
and on each great circle in the collection
$\{\xi_{k/n}\}_{k \in \Z} = \grp{G}_n(\Phi^{-1}(\Sp^1))$.
These last $n$ great circles intersect in pairs
only at $\pm p$,
so in total we have at least
$n(2n-2)+2 = 2n^2-2n + 2$ distinct umbilics,
each of multiplicity at least $n-2$.
Thus, counting with multiplicity,
the number of umbilics is at least
$(n-2)(2n^2-2n+2) = 2n^3 - 6n^2 + 6n -4$
and, by Lemma~\ref{lem:umbilics}\,\ref{lem:umbilics:count},
at most $8n-4$, requiring $n \leq 3$, and completing the analysis of Case \ref{case1}.

\item\label{case2} $\Phi(\Sp^1)$ intersects $\Sp^1$ nonorthogonally.
Let $\pm q\in\Sp^1$ be the two antipodal points of intersection.
In fact we must then have at least three great circles 
$\Sp^1$, $\gamma_1=\Phi(\Sp^1)$, and $\gamma_2 \vcentcolon= \refl{\Sp}^1(\gamma_1)$ on $\Gamma_n$
all intersecting at $\pm q$, and with neither $\gamma_1$ nor $\gamma_2$
orthogonal to $\Sp^1$. 
In particular $q$ is an umbilic (of multiplicity at least one),
and by the symmetries we then have at least $2n$ umbilics 
(each of multiplicity at least one)
on each of these last three great circles,
but these circles intersect only at $\pm q$,
so on these three circles we have at least
$3(2n-2)+2=6n-4$ distinct umbilics.
We recall from \eqref{eqn:Zn} that $\refl{\xi}_{1/n} \circ \refl{\xi}_0\in\grp{G}_n$ is a rotation of angle 
$2\pi/n$ along $\Sp^1$ and consider 
\begin{align*}
q'&\vcentcolon=(\refl{\xi}_{1/n} \circ \refl{\xi}_0)q,
&
\gamma_3&\vcentcolon=(\refl{\xi}_{1/n} \circ \refl{\xi}_0)\gamma_1,
&
\gamma_4&\vcentcolon=(\refl{\xi}_{1/n} \circ \refl{\xi}_0)\gamma_2,
\end{align*}
so that
$q'\notin \{q,-q\}$ and 
$\{q',-q'\} = \Sp^1 \cap \gamma_3 \cap \gamma_4$;
note also that neither $\gamma_3$ nor $\gamma_4$ intersects $\Sp^1$ orthogonally. 
Since $q'\neq \pm q$, each of $\gamma_3$ and $\gamma_4$
is distinct from each of the three circles
$\Sp^1$, $\gamma_1$, $\gamma_2$,
and they are also distinct from one another.
Moreover, each of the five circles in question
meets any other in exactly two antipodal points.
We conclude that on these five great circles we have
at least $(6n-4) + 2(2n-6) = 10n-16$ distinct umbilics
(of multiplicity at least one).

The circle $\Sp^1$ is disjoint from $\Sp^1_\perp$. 
None of the four great circles $\gamma_1,\ldots,\gamma_4\subset\Gamma_n$ 
can pass through the umbilics $\pm p$ on $\Sp^1_\perp$,
because if one did,
it would then intersect $\Sp^1_\perp$ orthogonally (because $\Gamma_n$ does so)
and therefore also intersect $\Sp^1$ orthogonally,
contradicting our assumption. 
Thus these last two umbilics are distinct
from the $10n-16$ identified above
and, moreover,
by reflection through any one of the $\gamma_i$
we identify two antipodal umbilics of multiplicity at least $n-2$
distinct from $\pm p$.
Counting multiplicity,
we then have at least
$(10n-16-2) + 4(n-2) = 14n - 26$ umbilics.
Appealing again to the upper bound $8n-4$ from Lemma~\ref{lem:umbilics}\,\ref{lem:umbilics:count}, 
we conclude $n < 4$,
completing the proof in case \ref{case2}.

\item\label{case3} $\Phi(\Sp^1) \cap \Sp^1 = \emptyset$.
Then also $\Phi(\Sp^1_\perp) \cap \Sp^1_\perp = \emptyset$ 
and $q\vcentcolon=\Phi(p)\notin\{p,-p\}$.
We distinguish the two cases wherein
(\ref{case3}.a) $q \in \Sp^1$
or (\ref{case3}.b) $q \not\in \Sp^1$.
In the first case the symmetries yield at least
$2n$ umbilics on each of the disjoint great circles
$\Sp^1$ and $\Phi(\Sp^1)$,
each of multiplicity at least $n-2$.
Using again Lemma~\ref{lem:umbilics}\,\ref{lem:umbilics:count},
we conclude $n<4$ in case (\ref{case3}.a).
In case (\ref{case3}.b) we consider the orbit 
$Q$ of $q$ under the subgroup
$\sk{\refl{\Xi}_{1/(2n)},\refl{\Xi}_{-1/(2n)}}$ of $\grp{G}_n$.
Since $q \notin \Sp^1 \cup \Sp^1_\perp$, the set 
$Q \cup \refl{\Sp}^1(Q)$ has cardinality $4n$,
so as in case (\ref{case3}.a)
we have at least $4n$ distinct umbilics of multiplicity at least $n-2$,
leading again to the conclusion that $n<4$, and ending the proof. \qedhere
\end{enumerate}
\end{proof}

\begin{proof}[Proof of Theorem~\ref{thm:main}\,\ref{thm:main-symmetry}]
Lemma \ref{lem:invariance-of-Sp1} and Lemma \ref{lem:syms_preserving_Sp1} imply  
$\symgrp{\Gamma_n}=\grp{G}_n$ for all $n\geq4$.
\end{proof}

\pagebreak[3]

\subsection{Morse index lower bound}
\label{subsec:index}

Savo's  \cite[Theorem 1.3]{Savo2010} general index estimate states that any minimal surface in $\Sp^3$ with genus $g\geq1$ has Morse index at least $g/2+4$.  
Thus, the surfaces constructed in Theorem~\ref{thm:main} have Morse index $\ind(\Gamma_n)\geq n+3$. 
By leveraging the prescribed symmetry group, we can significantly sharpen this general lower bound. 
Our result can be viewed as one possible spherical analogue of the free-boundary index estimate \cite[Proposition 4.2]{Carlotto2025} by Carlotto and the authors.
The proofs are both inspired by estimates of Montiel--Ros \cite[Lemmata 12--13]{MontielRos1991}
in the context of complete minimal surfaces of finite total curvature in $\R^3$,
especially as they were employed in \cite{KapouleasWiygul2020} by Kapouleas and the second author. 
An alternative approach could be based on the index lower bound \cite[Theorem~3]{ChoeVision} proved by Choe
rather than the Montiel--Ros methodology.

Recalling the notation introduced in \eqref{eqn:group}, we define the symmetry subgroup 
\begin{align*}
\grp{Y}_n
&\vcentcolon=\sk{\refl{\Xi}_{\frac{1}{2n}}, ~ \refl{\Xi}_{-\frac{1}{2n}}}
    \subset\grp{A}_n \subset\grp{G}_n
\end{align*}
for each $2\leq n\in\N$. The group $\grp{Y}_n$ is isomorphic to the \emph{pyramidal group} of order $2n$ (see~\cite{FranzKetoverSchulz,Carlottob,Carlotto}).

\begin{proposition}[Index lower bounds under pyramidal symmetry]
\label{prop:pyramidal}
Given $2\leq n\in\N$ let $\Sigma\subset\Sp^3$ be any closed, embedded, $\grp{Y}_n$-equivariant minimal surface with genus at least two.
\begin{enumerate}[label={\normalfont(\roman*)}]
\item\label{prop:pyramidal2n-1} Then $\Sigma$ has Morse index
\(\ind(\Sigma) \geq 2n-1\).
\item\label{prop:pyramidal4n-1} If moreover $\Sigma \supset \Sp^1$, then in fact
\(\ind(\Sigma) \geq 4n-1\).
\end{enumerate}
\end{proposition}

\begin{proof}
Let $u$ be the Jacobi field on $\Sigma$ which is induced by the rotations along $\Sp^1$
(fixing $\Sp^1_\perp$ pointwise). 
Then $u$ is nontrivial, because otherwise $\Sigma$ would have a continuous symmetry
and so be a great sphere or Clifford torus.
Since $\Sigma$ is $\grp{Y}_n$-equivariant,
it has (at least) $n$ different great spheres of symmetry, which we denote by 
$\Xi_{(1+2j)/(2n)}$ for $j\in\Z$.
It follows that
$\Sigma \setminus \bigcup_{j \in \Z} \Xi_{(1+2j)/(2n)}$
consists of exactly $2n$  pairwise isometric 
connected components
$\Omega_1, \ldots, \Omega_{2n}$,
each a Lipschitz domain in $\Sigma$
on which $u$ is an eigenfunction of eigenvalue zero
for the Jacobi operator of $\Sigma$
subject to the homogeneous Dirichlet boundary condition.
The Montiel--Ros lower bound
\cite[Corollary~3.2\;(i)]{Carlotto2025}
(applied with $T$ the Jacobi operator, $G$ trivial,
and $2n$ in place of $n$)
then completes the proof of \ref{prop:pyramidal2n-1}.

Under the additional assumption $\Sp^1 \subset \Sigma$, the restriction of $u$ to any $\Omega_i$ vanishes on $\Omega_i \cap \Sp^1$, so it is not a first Dirichlet eigenfunction for the Jacobi operator on $\Omega_i$. 
Application of \cite[Corollary~3.2\;(i)]{Carlotto2025} concludes the proof.
\end{proof}

\begin{proof}[Proof of Theorem~\ref{thm:main}\,\ref{thm:main-index}]
The minimal surface $\Gamma_n$ satisfies the hypothesis of Proposition~\ref{prop:pyramidal}\,\ref{prop:pyramidal4n-1}. 
\end{proof}

\begin{remark}[Morse index upper bounds]
By \cite{Franz2023} (see also \cite[Theorem~2.14]{FranzKetoverSchulz}) the $\grp{G}_n$-equivariant index of $\Gamma_n$ is at most $1$. 
The presence of odd symmetries such as $\refl{\Sp}^1\in\grp{G}_n$ makes it harder 
to apply Montiel--Ros methods aiming at Morse index upper bounds as e.\,g.~in \cite{Carlotto2025,Carlottob}.  
We refer to \cite{Ejiri2008} for a general upper bound on the Morse index of a minimal surface depending on its area and genus. 
\end{remark}

\begin{conjecture}\label{conj:index}
For every $2\leq n\in\N$ the Morse index of $\Gamma_n\subset\Sp^3$ is equal to $4n+5$. 
\end{conjecture}

This conjecture is motivated by the structural analogy with the complete Costa--Hoffman--Meeks surfaces in $\R^3$. 
For any genus $g$, these surfaces are known to have Morse index equal to $2g+3$ by \cite{Morabito2009,Nayatani1993}. 
This value decomposes as $2(g+1)+1$, where $2(g+1)$ is half the order of their symmetry group and $1$ is the index of the complete catenoid. 
A structurally similar formular for the Morse index of their free boundary minimal analogues in the unit ball  has been proposed in \cite[Conjecture~7.6]{Carlotto}. 
Applying this heuristic to our spherical setting yields the formula asserted in Conjecture~\ref{conj:index}. 
Indeed, by Theorem~\ref{thm:main}\,\ref{thm:main-symmetry}, $4n$ is half the order of $\symgrp{\Gamma_n}$, at least for $n\geq4$, and $5$ is the Morse index of the Clifford torus by \cite{Urbano1990}. 
Our numerical simulations provide further support for this conjecture.

\subsection{Distinctness}
\label{subsec:distinctness}

We make a series of ad hoc observations to explain how the minimal surfaces obtained in Theorem~\ref{thm:main} can be distinguished rigorously (in terms of congruence of surfaces in $\Sp^3$) from other closed, embedded minimal surfaces in $\Sp^3$. 
We consider all the examples we have found in the literature.

\begin{definition}
Two subsets $\Sigma,\Sigma'\subset\Sp^3$ are called \emph{congruent} if there exists an ambient isometry $\Phi\colon\Sp^3\to\Sp^3$ such that $\Phi(\Sigma)=\Sigma'$. 
Otherwise, $\Sigma$ is called \emph{geometrically distinct} from $\Sigma'$. 
\end{definition}

\begin{proposition}[Lawson surfaces {\cite[§\,6]{Lawson1970}}]
\label{prop:Lawson}
For any $k,m\in\N$ there exists a closed, embedded minimal surface $\xi_{m,k}\subset\Sp^3$ with the following properties. 
\begin{enumerate}[label={\normalfont(\roman*)}]
\item\label{prop:Lawson-genus} $\xi_{m,k}$ has genus $mk$. 
The surfaces $\xi_{m,k}$ and $\xi_{k,m}$ are congruent.

\item\label{prop:Lawson-axes} $\xi_{m,k}$ contains $(m+1)(k+1)$ great circles and their union  intersects $\Sp^1_\perp$ orthogonally in $2m+2$ equidistant points, and 
intersects $\Sp^1$ orthogonally in $2k+2$ equidistant points. 
\item\label{prop:Lawson-sym} $\xi_{m,k}$
has $m+1$ great spheres of symmetry
containing $\Sp^1$ and their union intersects $\Sp^1_\perp$
orthogonally in $2m+2$ equidistant points. 
Moreover, $\xi_{m,k}$ has $k+1$ great spheres of symmetry
containing $\Sp^1_\perp$ and their union intersects $\Sp^1$
orthogonally in $2k+2$ equidistant points.
In particular, neither $\Sp^1$ nor $\Sp^1_\perp$ is contained in $\xi_{m,k}$ if $k\geq1$.

\item\label{prop:Lawson-extra}For $m \geq 2$ and $k \geq 1$ there are no other great spheres of symmetry than those stated in \ref{prop:Lawson-sym}. 
\end{enumerate}
\end{proposition}

Proposition~\ref{prop:Lawson}\,\ref{prop:Lawson-extra} can be seen from \cite[Lemma~3.10]{KW-lchar}, referring also to the proof in case $k=m$; in this reference $M=M[m,k]=\xi_{m-1,k-1}$.

\begin{lemma}[Distinctness from the Lawson surfaces]
\label{lem:Lawson}
For any $3\leq n\in\N$ the minimal surface $\Gamma_n\subset\Sp^3$ constructed in Theorem~\ref{thm:main} is geometrically distinct from any of the Lawson surfaces $\xi_{m,k}$. 
\end{lemma}

\begin{proof}
Suppose that $\Gamma_n$ and $\xi_{m,k}$ are congruent for some $k \leq m \in \N$.
Since the genus of $\Gamma_n$ is at least~$4$, we may assume $k \geq 1$ and $m \geq 2$ without loss of generality.
Furthermore $\Gamma_n$ has two great spheres of symmetry
intersecting along $\Sp^1_\perp$
at angle $\pi/n < \pi/2$.
Therefore, by Proposition~\ref{prop:Lawson}\,\ref{prop:Lawson-sym}--\ref{prop:Lawson-extra}, 
$\Sp^1 \not\subset \Gamma_n$ which contradicts Theorem~\ref{thm:main}\,\ref{thm:main-axes}. 
\end{proof}

\begin{lemma}[The case $n=2$]\label{lem:n=2}
If the minimal surface $\Gamma_2\subset\Sp^3$ constructed in Theorem~\ref{thm:main} is congruent to a Lawson surface $\xi_{m,k}$, then $m=k=2$.  
\end{lemma}

\begin{proof}
By Theorem~\ref{thm:main}\,\ref{thm:main-genus} the surface $\Gamma_2$ has either genus $4$ or genus $2$; thus it suffices to prove that it is geometrically distinct from $\xi_{4,1}$ and $\xi_{2,1}$, 
which we do by contradiction under the assumption $\xi_{m,k} = \Phi(\Gamma_n)$ for some $\Phi \in \Ogroup(4)$.
Since $\Gamma_2$ contains $\Sp^1 \cup \xi_0 \cup \xi_1$ and is 
$\refl{\Xi}_{\pm1/4}$-equivariant, 
we conclude that $\xi_{m,k}$ has two spheres of symmetry that intersect orthogonally along a great circle $\Phi(\Sp^1_\perp)$ which itself orthogonally intersects two great circles
$\Phi(\xi_0), \Phi(\xi_1) \subset \xi_{m,k}$
and whose orthogonally complementary great circle
$\Phi(\Sp^1)$ also lies on $\xi_{m,k}$.
It follows from Proposition~\ref{prop:Lawson}\,\ref{prop:Lawson-sym}--\ref{prop:Lawson-extra}
that $\Phi(\Sp^1_\perp)$ intersects $\Sp^1$ and $\Sp^1_\perp$ orthogonally
and that the circles $\Phi(\xi_0)$ and $\Phi(\xi_1)$
are not any of the ones listed in Proposition~\ref{prop:Lawson}\,\ref{prop:Lawson-axes}
but rather must lie on the Clifford torus $\T^2$, so that $\Sp^1$ and $\Sp^1_\perp$ are exchanged under each of $\refl{\xi}_0$ and $\refl{\xi}_1$.
Thus we have $m=k=2$.
\end{proof}

\begin{remark}\label{rem:n=2}
The Lawson surface $\xi_{2,1}$ is conjectured to be unique (up to congruence) in the class of embedded minimal surfaces of genus $2$ in $\Sp^3$ -- at least under certain symmetry assumptions. 
If such a uniqueness result were known under compatible symmetry assumptions, Lemma~\ref{lem:n=2} would rule out any topological degeneration during our min-max procedure in the case $n=2$ and we would obtain $\genus(\Gamma_2)=4$.
Numerical simulations suggest that $\Gamma_2$ is in fact congruent to $\xi_{2,2}$. 
This is remarkable because the full symmetry group of $\xi_{2,2}$ has order $144$, while the prescribed group $\grp{G}_2$ is only of order $16$. 
Figure \ref{fig:n=2} displays a corresponding tessellation of fundamental domains.
(It has been shown in \cite{KW-lchar} and \cite{KusnerLuWang}
that each $\xi_{m,k}$ can be characterized by its genus
and various symmetry subgroups,
but $\grp{G}_2$ is too small for either result to apply here.)
\end{remark}

\begin{figure}
\pgfmathsetmacro{\widthfactor}{0.85}
\pgfmathsetmacro{\wL}{0.6*\widthfactor}
\pgfmathsetmacro{\wR}{0.4*\widthfactor}
\includegraphics[width=\wL\textwidth]{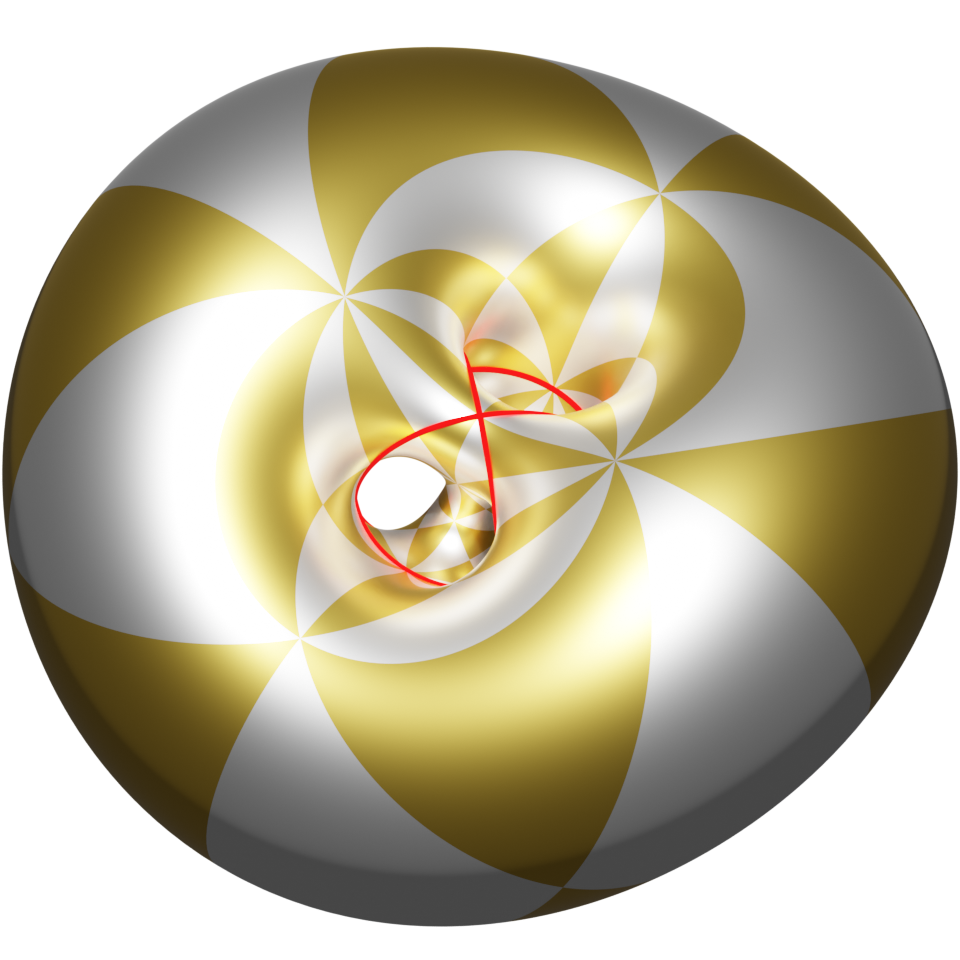}
\hfill
\includegraphics[width=\wR\textwidth]{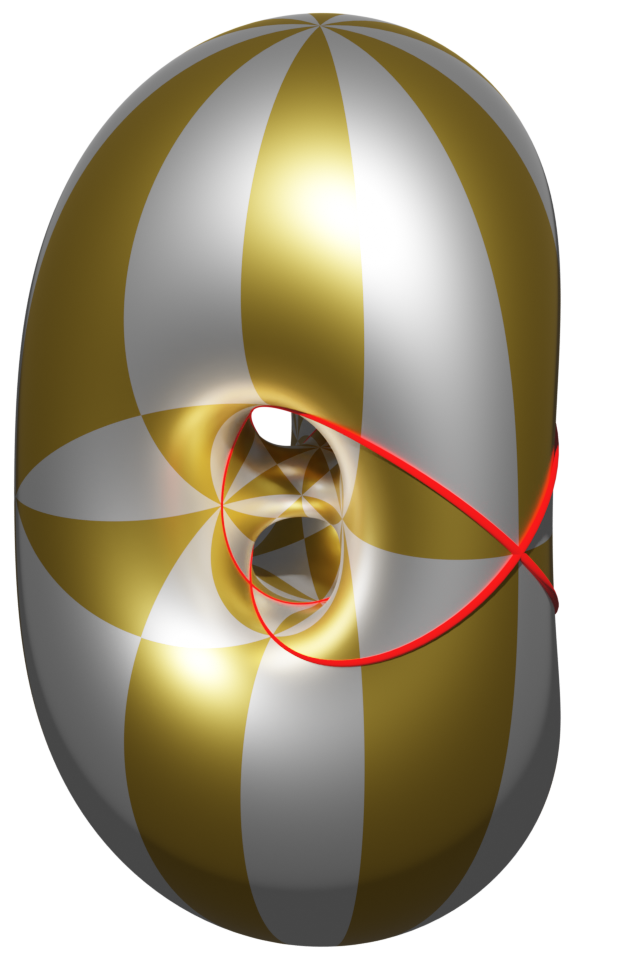}
\caption{Left image: Stereographic projection of $\Gamma_2$ containing the great circles $\Sp^1,\xi_0,\xi_{1/2}$. \\
Right image: Typical stereographic projection of $\xi_{2,2}$ and its $144$ fundamental domains.   
}%
\label{fig:n=2}%
\end{figure}

\begin{lemma}[Distinctness from examples of odd genus]
For any $2 \leq n \in \N$ the surface $\Gamma_n \subset \Sp^3$ constructed in Theorem \ref{thm:main} is geometrically distinct from all examples of odd genus.
\end{lemma}

\begin{proof}
Theorem~\ref{thm:main}\,\ref{thm:main-genus} asserts that $\Gamma_n$ has even genus.
\end{proof}

\begin{remark}[Some previously constructed and proposed examples of odd genus]
The following minimal surfaces in $\Sp^3$ have odd genus:
\begin{itemize}[nosep]
    \item the Karcher--Pinkall--Sterling \cite{KarcherPinkallSterling} examples,
  (except the one of genus $6$, cf.~Lemma~\ref{lem:Gamma3}),

    \item the torus desingularizations 
    constructed in \cite{ChoeSoret-tordesing,KW-tordesing,KW-tormore},

    \item the examples 3--11 proposed by Pitts and Rubinstein
      in \cite[Table 1]{PittsRubinstein1988},
      many of which have been constructed 
      by Ketover \cite{Ketover2016Equivariant},

    \item all examples of \cite{BaiWangWang},
  
    \item the numerical simulations with genus $11$ and $29$ in \cite{BobenkoHellerSchmitt-CMC}
    and every member of the family identified numerically in
    \cite[Proposition~4.3]{BobenkoHellerSchmitt-refl}.
\end{itemize}
\end{remark}

\begin{proposition}[Genus-six Karcher--Pinkall--Sterling surface]\label{prop:genus6KPS}
The minimal surface $\Sigma_6^{\mathrm{KPS}}$ of genus $6$ constructed in \cite{KarcherPinkallSterling} has 
has exactly $20$ umbilics, all of multiplicity one.
Each umbilic is an intersection point of exactly three great circles contained in $\Sigma_6^{\mathrm{KPS}}$. 
Let $x\in\gamma\subset\Sigma_6^{\mathrm{KPS}}$ be respectively such an umbilic and great circle and let
$y$ be either of the two points on $\gamma$
at distance $\pi/2$ from $x$.
Then $y$ is also an umbilic.
In particular there does not exist a pair of great circles 
on $\Sigma_6^{\mathrm{KPS}}$ intersecting orthogonally at $y$.
\end{proposition}

\begin{proof}
By construction, the surface $\Sigma_6^{\mathrm{KPS}}$ can be divided into $5$ pairwise isometric pieces (called ``bones'' in  \cite{KarcherPinkallSterling}) with boundary on a tetrahedral cell (see \cite[Figures~4--5]{KarcherPinkallSterling}) and each of them contains four interior points where three great circles contained in $\Sigma_6^{\text{KPS}}$ intersect.  
On the one hand, each of these points is an umbilic of multiplicity at least one and we have $20$ such points in total.
On the other hand the surface has genus $6$, so exactly $20$ umbilics
counted with multiplicity by \cite[Proposition~1.5]{Lawson1970}.
We conclude that $\Sigma_6^{\mathrm{KPS}}$ contains exactly $20$ umbilics,
each of multiplicity one, 
and only three great circles through each of them.
The final claim then follows
from the edge lengths, all $\pi/4$,
of the quadrilateral for which the Plateau problem is solved
in \cite{KarcherPinkallSterling} to generate $\Sigma^{\mathrm{KPS}}_6$.
\end{proof}

\begin{lemma}\label{lem:Gamma3}
The surface $\Gamma_3$ constructed in Theorem~\ref{thm:main} is geometrically distinct from $\Sigma_6^{\mathrm{KPS}}$. 
\end{lemma}

\begin{proof}
By construction, $\Gamma_3$ has an umbilic at $x=(0,0,1,0)$.
Take $y\in\Gamma_3$ such that $\xi_{\frac{1}{n}}\cap \Sp^1 = \{y,-y\}$.
Then $y$ lies at distance $\pi/2$ from $x$ along $\xi_{\frac{1}{n}}$,
which intersects $\Sp^1$ orthogonally at $y$. 
The claim then follows from Proposition \ref{prop:genus6KPS}.
\end{proof}

\begin{lemma}[Distinctness from examples of area $\geq12\pi$]
For any $2 \leq n \in \N$ the surface $\Gamma_n \subset \Sp^3$ constructed in Theorem \ref{thm:main} is geometrically distinct from all surfaces having area at least $12\pi$.
\end{lemma}

\begin{proof}
Theorem~\ref{thm:main}\,\ref{thm:main-area} implies that $\Gamma_n$ has area strictly below $12\pi$. 
\end{proof}

\begin{remark}[Stackings of $\T^2$ by gluing have area at least $12\pi$]
The gluing doublings (and stackings) of $\T^2$ in
\cite{KapouleasYang,W-torstack,KapouleasMcGrath2020,KapouleasMcGrath-tordbl}
all have area close to or greater than $2\hsd^2(\T^2)=4\pi^2 > 12\pi$.
\end{remark}

\begin{lemma}[Distinctness from examples containing no great circles]
For any $2 \leq n \in \N$ the surface $\Gamma_n \subset \Sp^3$ constructed in Theorem \ref{thm:main} is geometrically distinct from all surfaces which do not contain any great circles,
so in particular from any surface which divides $\Sp^3$ into unequal volumes. 
\end{lemma}

\begin{proof}
Theorem~\ref{thm:main}\,\ref{thm:main-axes}
implies that $\Gamma_n$ contains a great circle in $\Sp^3$,
reflection through which yields a congruence of the two components of $\Sp^3 \setminus \Gamma_n$.
\end{proof}

\begin{remark}[Doublings in $\Sp^3$ by gluing contain no great circles]
Each of the doublings of $\Sp^2$ by gluing presented in \cites{Kapouleas-sphdbl,KapouleasMcGrath-sphdbl,KapouleasMcGrath2020} divides $\Sp^3$ into two components one of whose volume is small. 
\end{remark}

\begin{lemma}[Distinctness from examples having a symmetry without any real eigenvalues]
For any $4\leq n\in \N$ the surface $\Gamma_n \subset \Sp^3$ constructed in Theorem \ref{thm:main} is distinct from every example admitting a symmetry that does not have any eigenvectors in $\R^4$. 
\end{lemma}

\begin{proof}
By Theorem~\ref{thm:main}\,\ref{thm:main-symmetry} 
we have $\symgrp{\Gamma_n}=\grp{G}_n$,
but it is clear that each generator of $\grp{G}_n$
preserves the set $\{(0,0,\pm1,0)\}$,
confirming that every element of $\symgrp{\Gamma_n}$
has either $1$ or $-1$ as an eigenvalue.
\end{proof}

\begin{remark}[Some examples having symmetries without a real eigenvalue]
The following minimal surfaces in $\Sp^3$ are symmetric with respect to helicoidal motions in $\Sp^3$ for which neither rotation angle is an integer multiple of $\pi$.
\begin{itemize}[nosep]
\item The $(\Z_m \times \Z_m)$-equivariant min-max doublings \cite[Theorem~3.6]{KetoverMarquesNeves2020} of $\T^2$ for $m \geq 3$,
\item the surfaces $\tilde{M}_{p,q} \subset \Sp^3$ constructed 
in \cite{KetoverFlipping} as lifts of certain minimal surfaces in lens spaces.  
\end{itemize} 
\end{remark}

\begin{lemma}[Distinctness from sphere doublings]
Let $4 \leq n \in \N$,
and let $\Xi \subset \Sp^3$
be a great sphere such that
$\refl{\Xi} \in \symgrp{\Gamma_n}$.
Then the closure of each component of $\Gamma_n \setminus \Xi$
is a connected surface with genus $\gamma$ and $\beta$ boundary components,
where
\begin{align*}
  \beta
  &=
    \begin{cases}
      1 &\text{if $\genus(\Gamma_n)=2n$ or $n \in 2\Z + 1$,}
      \\
      3 &\text{if $\genus(\Gamma_n)=2n-2$ and $n \in 2\Z$,}
    \end{cases}
\\
  \gamma
  &=
    \begin{cases}
      n &\text{if $\genus(\Gamma_n)=2n$,}
      \\
      n-1 &\text{if $\genus(\Gamma_n)=2n-2$ and $n \in 2\Z+1$,}
      \\
      n-2 &\text{if $\genus(\Gamma_n)=2n-2$ and $n \in 2\Z$.}
    \end{cases}
\end{align*}
In all cases we have $\gamma > 0$, so $\Gamma_n$ is geometrically distinct from every surface $\Sigma$ having a great sphere of symmetry $\Xi$ such that each component of $\Sigma \setminus \Xi$ has genus zero.
\end{lemma}

\begin{proof}
Theorem~\ref{thm:main}\,\ref{thm:main-symmetry} implies $\Xi = \Xi_{(2i+1)/(2n)}$ for some $i\in\Z$. 
In particular $\Sp^1_\perp$ is contained in $\Xi$ and divides it into two hemispheres,
whose closures we call $\Xi^+$ and $\Xi^-$ and each of which is one of the $2n$ hemispherical components
that are the subject of Corollary \ref{cor:hemispheres}.
In case $\Gamma_n$ has genus $2n$,
each of $\Gamma_n\cap\Xi^\pm$
consists of a single arc with the same two endpoints on $\Sp^1_\perp$,
so it is obvious that $\beta=1$. 
Now let $\Phi \in \Ogroup(4)$ be given by
\begin{equation*}
\Phi\vcentcolon=
  \begin{cases}
    \refl{\Xi}_{\frac{i}{n}+\frac{1}{2}} &\mbox{if } n \in 2\Z
    \\
    \refl{\xi}_{\frac{i}{n}+\frac{1}{2}} &\mbox{if } n \in 2\Z + 1,
  \end{cases}
\end{equation*}
so that $\Phi \in \symgrp{\Gamma_n}$ and $\Phi(\Xi^+)= \Xi^-$.
That we also have $\beta=1$ for every odd $n$
is now clear from the final statement of
Corollary \ref{cor:hemispheres}.
If instead $n$ is even and $\Gamma_n$ has genus $2n-2$,
then $\Xi^+ \cap \Gamma_n$ consists of three arcs
with endpoints on $\Sp^1_\perp$,
whose union with their mirror images under $\Phi$ in this case
consists of three closed curves. 
The corresponding values of $\gamma$ then follow from the relation
\(
\genus(\Gamma_n)=2\gamma + \beta - 1
\), 
completing the proof. 
\end{proof}
 
\begin{remark}[Basic reflection surfaces]
All of the minimal surfaces in $\Sp^3$ constructed by eigenvalue optimization in \cite{KarpukhinKusnerMcGrathStern} can be cut by a great sphere of symmetry into a pair of genus-zero surfaces.
\end{remark}

\begin{lemma}[Distinctness from examples with too many spheres of symmetry]
\label{lem:too_many_spheres}
Let $4 \leq n \in \N$.
Then $\Gamma_n$ is distinct from every example
whose genus $g$ and number $N$ of spheres of reflection
satisfy $g < 2N-2$.
\end{lemma}
\begin{proof}
$\Gamma_n$ has genus $g \geq 2n-2$ and exactly $n$ great spheres of symmetry by 
Theorem~\ref{thm:main}\,\ref{thm:main-genus},\,\ref{thm:main-symmetry}.
\end{proof}
\begin{remark}[Remaining families of \cite{BobenkoHellerSchmitt-refl}]
Two infinite families of examples
$\{B_{k,1}\}_{k=3}^\infty$ and $ \{B_{2,\ell}\}_{\ell=3}^\infty$
are described in \cite[Propositions 4.1--4.2]{BobenkoHellerSchmitt-refl}
respectively,
where in particular it is stated that
$B_{k,1}$ has genus $k$
and $B_{2,\ell}$ genus $\ell+1$
and that each is invariant under (at least) a certain group,
which (as implied by the fourth bulleted assertion
in the list preceding \cite[Theorem 1.16]{BobenkoHellerSchmitt-refl})
in the case of $B_{k,1}$ includes $k+1$ reflections through spheres
and in the case of $B_{2,\ell}$ rather $\ell+2$ reflections through spheres.
In both cases we then have $g < 2g \leq 2N-2$
(in the notation of Lemma \ref{lem:too_many_spheres}).
\end{remark}


\setlength{\parskip}{1ex plus 2pt minus 2pt}
\bibliography{fbms-bibtex}

\printaddress

\end{document}